\documentclass[10pt]{article}

\makeatletter
\DeclareRobustCommand{\qed}{%
  \ifmmode % if math mode, assume display: omit penalty etc.
  \else \leavevmode\unskip\penalty9999 \hbox{}\nobreak\hfill
  \fi
  \quad\hbox{\qedsymbol}}
\newcommand{\openbox}{\leavevmode
  \hbox to.77778em{%
  \hfil\vrule
  \vbox to.675em{\hrule width.6em\vfil\hrule}%
  \vrule\hfil}}
\newcommand{\qedsymbol}{\openbox}
\newenvironment{proof}[1][\proofname]{\par
  \normalfont
  \topsep6\p@\@plus6\p@ \trivlist
  \item[\hskip\labelsep\itshape
    #1.]\ignorespaces
}{%
  \qed\endtrivlist
}
\newcommand{\proofname}{Proof}
\makeatother

\usepackage[utf8]{inputenc}

\usepackage{bm}

\usepackage{color}
\usepackage{latexsym}
\usepackage{dsfont}
\usepackage{amssymb}
\usepackage{comment}
\usepackage{graphicx}
\usepackage{amsmath,amsfonts,amssymb,enumerate,theorem,euscript,array,amsfonts,mathrsfs}
\usepackage{hyperref}
\usepackage{appendix}
\usepackage[T1]{fontenc}
\usepackage{babel}
\numberwithin{equation}{section}
\usepackage{bbm}
\usepackage{subfigure}
\usepackage{color}

\usepackage{stmaryrd}

\usepackage[algo2e,ruled,vlined]{algorithm2e} 
 \usepackage{algorithm}
 \usepackage{algorithmicx}
\usepackage{algpseudocode}
\usepackage{ marvosym }
\usepackage{hyperref}
\usepackage{bm}
\usepackage{xcolor, soul}

\usepackage{comment}

\usepackage{tikz}
\usetikzlibrary{fit,matrix,chains,positioning,decorations.pathreplacing,arrows}

\usepackage[a4paper,margin=1in,heightrounded]{geometry}

\usepackage{algpseudocode}
\usepackage{algorithm}

\usepackage[normalem]{ulem}

\def \b1{\bf{1}}

\def \R{\mathbb{R}}

\def \E{\mathbb{E}}
\def \F{\mathbb{F}}

\def \P{\mathbb{P}}

\def\esssup_#1{\underset{#1}{\mathrm{ess\,sup\, }}}

\def\argmin_#1{\underset{#1}{\mathrm{argmin\, }}}
\def\argmax_#1{\underset{#1}{\mathrm{argmax\, }}}

\def\dm#1{\frac{\delta}{\delta m}}

\def \Ac{{\cal A}}
\def \Bc{{\cal B}}

\def \Ec{{\cal E}}
\def \Fc{{\cal F}}

\def \Ic{{\cal I}}

\def \Pc{{\cal P}}

\def \Oc{{\cal O}}

\def \eps{\varepsilon}

\def\reff#1{{\rm(\ref{#1})}}

\def\beqs{\begin{eqnarray*}}
\def\enqs{\end{eqnarray*}}
\def\beq{\begin{eqnarray}}
\def\enq{\end{eqnarray}}

\newtheorem{Theorem}{Theorem}[section]

\newtheorem{Proposition}{Proposition}[section]

\newtheorem{Lemma}{Lemma}[section]

\numberwithin{equation}{section}

\title{%A Martingale Optimality Approach \\for 
Portfolio Exponential Utility Maximization\\ with Jump Signals}
\author{Lokmane Abbas Turki\footnote{LPSM, Sorbonne Université, \sf {lokmane.abbas$\_$turki at sorbonne-universite.fr}. This author partially benefited from the support of the Chair \textit{Capital Markets Tomorrow: Modeling and Computational Issues} under the aegis of the Institut Europlace de Finance,  a joint initiative of Laboratoire de Probabilit\'es, Statistique et Mod\'elisation (LPSM), Universit\'e Paris Cit\'e and  Cr\'edit Agricole CIB.}  \and Sigui Brice Dro\footnote{BPCE SA and LPSM, Sorbonne Université,  \sf{dro at lpsm.paris}.}  \and Idris Kharroubi \footnote{LPSM, Sorbonne Université, {\sf kharroubi at lpsm.paris}. The work of this author is  partially supported by Agence Nationale de la Recherche (ReLISCoP grant ANR-21-CE40-0001)}  }

\begin{document}

\maketitle

\begin{abstract}
In this paper, we study the portfolio utility maximization in the case where the risky asset is driven by a Brownian motion and an independent  homogeneous Poisson measure, with strategies that may include jump signals.   This means that the allowed strategies are no longer predictable but also include the information given by a process driven by the Poisson measure. Using the results of Bank and K\"orber \cite{bank2022merton}, we first express the considered portfolio as semi-martingale processes. We then present the martingale optimality principle for the exponential utility maximization.  
This allows to derive an original BSDE with jumps and to  express the optimal value and an optimal strategy using the solution to this original BSDE. We then prove existence of a solution to the considered BSDE. We finally present some numerical experiments to quantify the gain of utility given by the information from the jump signals.  
\end{abstract}

\tableofcontents{}
\newpage
\section{Introduction}
Finding a relevant investment strategy is a fundamental task in finance and has therefore attracted a lot of interest from practitioners and  academicians. Among this last community, the high frequency of transactions in financial markets leads to develop continuous time models in Brownian information going back to  Merton  \cite{Merton69} and \cite{Merton71} who studied combined investment/consumption choice.

Another problem close to the previous one and which has been considered by numerous works is the maximization of the portfolio terminal expected utility. Such a problem is especially interesting in the case where the considered strategies are subject to constraints as the model loses its financial completeness. One then needs to specify a valuation of contingent claims and  the indifference pricing which involves utility maximization is a possible solution.

Such a utility maximization problem under portfolio constraints has been considered  for the exponen--tial utility case by Cvitanic and Karatzas \cite{CK92}. Assuming that the portfolio constraints are convex, the authors solve the problem using convex duality.  
The case of general constraints is considered by Hu \textit{et al.} \cite{hu2005utility}. In this paper, the authors use a so-called martingale optimality principle to characterize an optimal solution and the optimal value. This martingale optimality principle allows to derive a Backward Stochastic Differential Equation (BSDE for short) that has a quadratic growth. The authors solve this BSDE and get an optimal solution. We mention that the general case of Brownian BSDEs with quadratic growth has been consider by Kobylansky \cite{kobylanski}. 

\vspace{2mm}

The literature has also considered the mixed Brownian-Poisson case. Becherer \cite{becherer2006bounded} considers the case where the contingent claim derives from both a Brownian motion and an independent jump Poisson process, but the underlying assets are continuous and the investment strategies are not constrained. The exponential utility maximization leads  to solve a BSDE with jumps that  has linear growth in the unknown, except the jump component. % $y$ and $z$ but not in $u$. 
The  case where the underlying assets are driven by both a Brownian motion and a Poisson jump measure, and some constraints are imposed on the investment strategies has also been considered.  Morlais \cite{MAM08} studies the case of  a finite activity jump and extends the study to the infinite activity case in\cite{MAM09}. The author also follows the martingale optimality principle which lead to solve a BSDE with jumps whose generator is quadratic in the Brownian stochastic integrand (usually denoted by $z$) and with an exponential term involving the Poisson stochastic integrand variable (usually denoted by $u$). The derived BSDE is solved using a penalization approach and relying on an extension of a stability results for BSDEs proved by Kobylansky \cite{kobylanski}  in the Brownian case. More generally, quadratic BSDEs with jumps has  attracted a lot of interest and many works study their solvability. Among them, Kazi-Tani \textit{et al.} \cite{KTPZ15}  solves the BSDE in the case where either the terminal condition is small enough or the driver is regular enough, El Karoui \textit{ et al.}\cite{EMN16} considers the case of a square integrable jump measure and Matoussi and Salhi \cite{MS19} 
supposes that the driver is separable  in the variables $z$ and $u$.

An interesting case of application in finance is the situation where, even if the investor faces jumps in the assets dynamics, some signal revealed by these jumps can be used to construct investment strategies.  This case has been considered by Bank and Korber \cite{bank2022merton}. The allowed strategies are no longer predictable but may be also measurable with respect to the $\sigma$-algebra generated by a process that is an integral of a given function with respect to the random Poisson measure. Unfortunately, the stochastic integral of those processes w.r.t. the Poisson measure cannot be defined by the classical theory of jump processes since the jump measure is suppose to have an infinite activity. To study the utility maximization for such an extended set of allowed strategies,  the authors first show that the stochastic integral can be generalized to the considered strategies but with a modified semimartingale decomposition involving the behavior of the jump assets given the value of the signal. This original semimartingale decomposition allows to formally derive the related HJB equation related to the utility maximization problem and the solution in the power utility case is discussed. We mention that this approach has been extended to the interesting case of a multi-agent model  by Bank and Sedrakjan \cite{BS25} in the case of a finite jump activity.       
 
%$\bullet$ Then turn to the case of signal initiated by Bank and Korber \cite{bank2022merton}. Explain the problem. Cite also its extention to equilibrium of several agents also by Bank and Sedrakjan \cite{BS25}.

\vspace{2mm}

We propose in this work to study the exponential utility maximization of bounded claims in the framework of strategies including signals as considered by Bank and Korber \cite{bank2022merton}.  

The case of an exponential utility function with signals  has not been studied by the previous works  and our goal is to provide a complete solution to this problem by relying on the martingale optimality principle and BSDEs in this framework of jump signals and infinite activity.

We first present a description of the market model and we recall the extension of the stochastic integral with respect to the jump measure for strategies which may not be predictable by following Bank and Korber \cite{bank2022merton}.  This extended integral  allows in particular to define self-financing portfolios for strategies including an additional information signal. 
We next present the related utility maximization problem which involves the previously defined portfolios. 

We then use the martingale optimality principle to solve this problem. %After describing the martingale optimality principle, we look for its solution using a BSDE.
The conditions imposed by the martingale optimality principle leads to an original BSDE whose driver is quadratic in the Brownian  integrand variable  $z$  and with an exponential term involving the Poisson  integrand variable  $u$. This BSDEs looks like that studied By Morlais in \cite{MAM08,MAM09}. However,  the dependence on the signal makes the driver more tricky as it has two components. The first one is related to the case where the signal vanishes and is similar to that of \cite{MAM08,MAM09}. The second part is related to the case where the signal indeed appears and involves the conditional jump intensity of the asset given the signal  is not vanishing. 

Both of the two terms in the driver are infimum of given functions.  As we usually have for BSDEs related to utility maximization problems, if we suppose the existence of a measurable argument infimum, an optimal strategy is given by this measurable argument infimum function applied to the solution of the BSDE.

\vspace{2mm}

 In our case, the specific form of the driver provides some strict convexity property which allows to get such a measurable argument infimum function and hence an optimal strategy as soon as we are able to solve the considered BSDE.

\vspace{2mm}
 
%We then use the martingale optimality principle to solve this problem. After describing the martingale optimality principle, we look for its solution using a BSDE. The conditions imposed by the martingale optimality principle leads to an original BSDE whose driver is quadratic in the Brownian stochastic integrand variable (usually denote by $z$) and with an exponential term involving the Poisson stochastic integrand variable (usually denote by $u$). This BSDEs looks like that studied  By Morlais in \cite{MAM08,MAM09}. However,  the dependence on the signal makes the driver more tricky as it has two components. The first one is related to the case where the signal vanishes and is similar to that of \cite{MAM08,MAM09}. The second part is related to the case where the signal indeed appears. To derive this term, we need to use the conditional jump intensity of the asset given the signal non vanishing. 

Unfortunately, the BSDE we get does not fall in any of the studied framework and needs to be solved. For that we follow the classical penalization approach. We first truncate the jump measure in the neighborhood of $0$ to get a finite measure. We also truncate the driver with respect to the variables $z$ and $u$. For the variable $z$, we introduce a classical regular and symmetric truncation function to transform the quadratic term into a Lipschitz function. Contrary to  \cite{MAM08,MAM09}, we cannot use the same truncation for the variable $u$ as it does not allow to get a monotone approximation with respect to the truncation parameter. We therefore consider another truncation function which is regular but  has a truncation effect only from below.   

We then provide some uniform estimate on the penalized solution and use the stability result of Morlais \cite{MAM08} to pass to the limit and get a solution to the BSDE.
%We also provide uniqueness of the solution to the considered BSDE. 

Finally, we numerically illustrate the effect of the additional information given by the signal. For that we compute the optimal values with and without signal on two examples. The first one is the case where the signal gives information only for small jumps. This can be related to the situation where an investor knows everything about  small troubles of a given asset, and try to take advantage of this information to  anticipate big troubles. The second case  is the reverse situation: the investor does not know anything about small troubles but has information only about big ones. This second situation is maybe most spread as for big troubles of companies, the information is usually revealed by e.g. newspapers, whereas small troubles as usually hidden by managers to avoid a depreciation of the assets.

\vspace{2mm}

The rest of the paper is organized as follows.
In Section \ref{sect2}, we present the probabilistic model of market. In Section \ref{sect3}, we present the set of classical strategies and the extended set of signal strategies. We also recall the extension of the stochastic Poisson  integral for signal strategies and the related semi-martingale decomposition and define the related self-financing portfolios.
In Section \ref{sect4}, we present the portfolio exponential utility maximization, and provide the related martingale optimality principle.
Section \ref{sect5} is dedicated to the derivation of the related BSDE and to the proof of the existence of a solution. An optimal strategy is also derived from the solution to the BSDE.
Finally, Section \ref{sect6} presents some numerical experiments to measure the effect of the information brought by the jump signal.
%Finally, the proof a relegated to the appendix in Section \ref{sect7}. 
\section{Market model}\label{sect2}
\subsection{Probabilistic settings}
We consider a complete probability space \((\Omega, \mathcal{A}, \mathbb{P})\) endowed with a standard one dimensional Brownian motion \((W_t)_{t \geq 0}\) and  an independent  homogeneous Poisson measure $N$ defined on  $\mathbb{R}_+\times \mathbb{R}$. We suppose that the compensator of $N$ is \(dt \otimes \nu(de)\), where $\nu$ is a  \(\sigma\)-finite measure on $(\mathbb{R}, \mathcal{B}(\mathbb{R}))$ such that
\begin{equation*}
    \nu(\{0\}) = 0
    \qquad\text{and}\qquad
    \int_{\mathbb{R}}(1\wedge|e|^2)\,\nu(de) < +\infty.
\end{equation*}
We denote by $\tilde{N}$ the compensated random measure of  $N$ and we recall that it is given by 
\[
\tilde N(dt,de) = N(dt,de) - \nu(de)\,dt.
\]
%We assume that the stochastic processes \((W_t)_t\) and \((N_t)_t\) are independent.
We define the  filtration \(\mathbb{F}=(\mathcal{F}_t)_{t \geq 0}\)  as the completion of the natural filtration generated by $W$ and $N$. $\mathbb{F}$  is therefore given by
\beqs
\mathcal{F}_t = \sigma(W_s, N([0,s] \times A) : s \leq t, A \in \mathcal{B}(\mathbb{R}))\vee \mathcal{N}\;,\quad t\geq 0\;,
\enqs
where $\mathcal{N}$ stands for the class of $\mathbb{P}$-negligible sets, and $\mathbb{F}$ 
and satisfies the usual conditions (i.e. right-continuity and $\mathcal{F}_0$ contains the $\mathbb{P}$-null sets).

Let \(\mathcal{P} = \sigma((Y_t)_t: Y \text{ is a continuous, } \mathbb{F}\text{-adapted process})\) denote the \(\sigma\)-algebra of predictable processes.

\subsection{The financial market}

We consider a financial market over a time period $[0,T]$  where $T>0$ is the deterministic time horizon. We suppose that this financial market is composed by two assets. The first one is a riskless  asset that we assume for simplicity to have a zero interest rate. 
%
%
%driven by a riskless asset $S^0$ at the constant rate $r \in \mathbb{R}$ and a risky asset $S^1$ driving both by the Brownian motion $W$ and the Poisson measure $N$. \\Their dynamic over some finite horizon $[0,T]$ is as follow
%\begin{itemize}
%    \item The first asset is a riskless asset $S^0$ with initial value $s_0>0$ and constant interest rate $r>0$. Its price process $(S^0_t)_{t\in[0,T]}$ satisfies the following SDE
%    \beqs
%S^{0}_{t} & = & s^{0}+\int_0^t rS_{u-}^{0}du, \quad  t \in [0, T]\;.
%\enqs
 % $s^0$ and $r$ are two constants representing the initial value of $S^0$ and the interest and satisfy $s^0>0$ and $r\geq0$.   
    %\item 
    The second asset is a risky asset denoted by $S^{}$. Its price process $(S_t)_{t\in[0,T]}$ satisfies the following SDE
\beqs
S^{}_t & = & s^{}_0+ \int_0^tS^{}_{u-}[ \kappa du + \sigma dW_u + \int_{\R}\eta(e) \tilde{N}(du, de)],\qquad  t\in[0, T]\;.
\enqs
The constants $s_0>0$, $\kappa \in \mathbb{R}$  and  $\sigma >0$ represent respectively the initial value of $S$, the drift and the volatility of the asset $S$. The function $\eta: \R\xrightarrow{}\mathbb{R}$ is a Borel map  specifying the jumps of the risky asset. We suppose that there exists a constant $C$ such that 
\beq\label{condetanu1}
|\eta (e)| & \leq &  C(1\wedge |e|)
\enq
for all $e\in\R$. Under the previous condition, we have existence and uniqueness of the processes $S$ (see e.g. \cite[Chap.  III Theorem 2.32]{jacod2013limit}). We assume in addition that $\nu(\{\eta<-1\}) = 0$ to ensure that the process $S$ remains nonnegative.

For later use, we shall also assume that 
\beq
\int_{\R}|\eta(e)|  \nu(de)  & < & \infty.\label{condetanu2}
\enq
  
%\end{itemize}

 %and $$ \nu(\{|\eta| >1\}) = 0 \text{ (for simplicity in the computations).}$$
%\red{donner ref existence unicité sol}
\section{Investment strategies and related wealth}\label{sect3}
\subsection{Classical admissible strategies}
To describe the investor strategy, we denote by $\pi_t$  the wealth invested in the risky asset $S$ at time $t\in[0,T]$. % of the investor's investment strategies and take values from a compact set $\mathcal{C}$ containing 0.
We consider  strategies  $\pi=(\pi_t)_{t\in[0,T]}$ that are $\Pc$-measurable processes and satisfy 
\beq\label{cond-Adm-pi}
%\mathbb{E}(
\int_{0}^{T}|\pi_{t}|^2dt %+ \int_{0}^{T}\int_\R |\pi_t \eta(e)|^2\nu(de)
%)
  & < &  +\infty\;,\quad  \P-a.s.
\enq
A strategy $\pi=(\pi_t)_{t\in[0,T]}$ that is $\Pc$-measurable and  satisfies \eqref{cond-Adm-pi} is said to be admissible. We denote by $\Ac$ the set of admissible strategies.
For a given initial wealth $x\in \R$ and an investment strategy $\pi=(\pi_t)_{t\in[0,T]}\in \Ac$, we denote by $X^{x,\pi}_t$ the resulting self financing wealth at time $t\in[0,T]$. The process $X^{x,\pi}$ satisfies the following self-financing dynamics
\beq\label{defDynX}
X^{x,\pi}_t  =  x+\int_0^t\frac{\pi_u}{S_{u-}}dS_u%+\int_0^t\frac{X^{x,\pi}_{u-}-\pi_u}{S_{u-}^0}dS^0_u 
\;,\quad t\in[0,T]\;.
\enq
We notice that $X^{x,\pi}$ is well defined for $\pi\in \Ac$ according to condition \eqref{cond-Adm-pi}. 
%Denote by $\tilde X^{x,\pi}$, $\tilde S$ and $\tilde{\pi}$ the discounted values of $X^{x,\pi}$ ,$S$ and $\pi$ respectively. A standard application of It\^o's formula shows that 
%\beqs
%\tilde X^{x,\pi}_t & = & x+\int_0^t\frac{\tilde{\pi}_u}{\tilde S_{u-}}d\tilde S_u \\
% & = &  x+\int_0^t \tilde{\pi}_u[(\kappa-r) du + \sigma dW_u + \int_{\R}\eta(e) \tilde{N}(du, de)] 
%\enqs
%for $t\in[0,T]$. We notice that $\tilde X^{x,\pi}$ and $X^{x,\pi}$ are well defined for $\pi\in \Ac$ according to condition \eqref{cond-Adm-pi}. 
%Without loss of generality and for simplicity, we will take $r=0$ in the rest of paper.\\
%In this case, $\pi$ is measurable with respect to $\mathcal{P}$ and the resulting self-financing portfolio of the investor $(X_{t}^{x, \pi})_t$ evolves according to :
%\beqsX_{0}^{x, \pi} = x >0
%,
%\enqs
%\beqs
%dX_{t}^{x, \pi} = \pi_{t}\frac{dS^{1}_{t}}{S^{1}_{t-}} = \pi_{t}[\kappa dt + \sigma dW_t + \int_{\R}\eta(e) \tilde{N}(dt, de)]
%\enqs

\subsection{Portfolio with signal strategies} We consider investment strategies in the case where the investor has access to an extra information on the stock $S$.  More precisely, we suppose that
the extra information on the stock is given by an impending signal process $G$ defined by
\beqs
G_{t}  & = & \int_{0}^{t} \int_{\R}\gamma(e) \tilde{N}(ds, de)\;,\quad t \geq 0\;,
\enqs
where
 $\gamma:\R \xrightarrow{}{\mathbb{R}}$ is a Borel map such that there exists a constant $C$ satisfying 
 \beqs
% \int_{\R}(\gamma(\R)^2+ |\gamma(\R)| )\nu(de) & < & \infty.
|\gamma(e)| & \leq & C(1\wedge |e|)
 \enqs
 for all $e\in \R  $. We also assume that $\gamma(\R)$ is a Borel subset of $\R$. %\textcolor{red}{We shall assume that $\gamma(\R)$ is an open subset of $\R$.}
 
We now consider the set  of strategies including this additional information given by the signal. Namely, we would like to consider strategies 
\beqs
\pi=(\pi_t)_{t\in[0,T]} \text{ that are }\mathcal{P} \vee \sigma(G)\text{-measurable and satisfy  } \eqref{cond-Adm-pi}\;.
\enqs
%$\pi=(\pi_t)_{t\in[0,T]}$ which are $\mathcal{P} \vee \sigma(Z)$-measurabe
%In this second case, the strategies $\pi$ are allowed to be mesurable with respect to the larger $\sigma-$ field $$\mathcal{P} \vee \sigma(Z).$$\\ 
By \cite[Lemma 2.2]{bank2022merton} a process $\pi$ is $\mathcal{P} \vee \sigma(G)$-measurable if and only if there exists a map $p:~\R_+\times\Omega\times\R\rightarrow\R$ which is 
 $\mathcal{P} \otimes \mathcal{B}(\mathbb R)$-measurable, such that
\beqs
\pi_t & = & p_t(\Delta_t G)\;,\quad t\geq0 ~\P-a.s.
\enqs
where $\Delta_t G=G_t-G_{t-}$ for $t\geq0$.
In particular, we can rewrite the set of such strategies as
\beqs
 \text{Processes }(\pi_t=p_t(\Delta G_t))_{t\in[0,T]} \text{ satisfying  } \eqref{cond-Adm-pi} \mbox{ with } p  \text{ being }\mathcal{P} \otimes \Bc(\R)\text{-measurable} \;.
\enqs
For such a strategy $\pi$, we would like to define a portfolio value process $X^{x,\pi}$ satisfying \eqref{defDynX}. To do this, we need to define the stochastic integral of $\pi$ w.r.t. $N$. Such a stochastic integral is indeed only defined for $\Pc$-measurable processes, whereas, the considered $\pi$ has extra measurability.

We define the kernel $K:~\R\times \Bc(\R)\rightarrow \R $ such that 
\begin{equation}
\nu\big(de \cap \{\gamma \neq 0\}\big) = \int_{\gamma(\R) \setminus\{0\}} K(g,de)\mu(dg),  
\end{equation}
 where $\mu$ is the measure image of $\nu $ by $\gamma$, $i.e.$ $\mu = \nu \circ \gamma^{-1}$. To get such a kernel $K$, we consider the measure $M$ defined on $\Bc(\R)\otimes \Bc(\R)$ by
 \begin{equation}
M(A\times B)= \nu\big(\gamma^{-1}(A\setminus \{0\})\cap B\big)
 \end{equation}
 for any $A\in \Bc(\R)$ and $B\in\Bc(\R)$. Applying \cite[Chapter II, Paragraph 1.2]{jacod2013limit}, we get the existence of such a kernel $K$ as the first marginal of $M$ is  $\mu\mathds{1}_{\gamma(\R)\setminus\{0\}}$. Moreover, we have $K(g,\R)=K(g,\{\gamma=g\})=1$ for all $g\in \R\setminus\{0\}$.

 Let us then introduce the average jump size $\hat \eta(g)$ given a signal $g$ and defined by  
 \beqs
 \hat{\eta}(g) = \int_{\{\gamma = g\}} \eta(e) K(g, de) 
 \enqs
 for $g \in \gamma(\R) \setminus{\{0\}}$. %
% with the convention $\hat{\eta}(g) =0$ for $g$ such that $K(g,\R)=0$.
We also define the related variance function $v_\eta$ by
\begin{equation}
v_{\eta}(g)  =  \int_{\{\gamma = g\}} |\eta(e)-\hat \eta(g)|^2 K(g, de) 
\end{equation}
for $g\in \gamma(\R) \setminus{\{0\}}$.
We notice that \reff{condetanu1} and \reff{condetanu2} implies that  $\hat \eta$ is well defined and satisfies
\beq\label{condAdmContSignal1}
\int_{\gamma(\R) \setminus\{0\}} |\hat \eta(g)|\mu(dg) & < & +\infty\;, \\ \label{condAdmContSignal2}
\int_{\gamma(\R) \setminus\{0\}} v_\eta(g)\mu(dg) & < & +\infty\;.  
\enq
These conditions allow to define the stochastic integrals of $\Pc\vee \sigma(G)$-measurable processes w.r.t. $\tilde N$ as done by \cite[Theorem 2.4]{bank2022merton}. We present a version of this result adapted to our framework. To this end, we introduce the following integrability condition
\beq\label{cond-loc-smg}
\int_0^T\left(p_s(0)^2+\int_{\gamma(\R)\setminus\{0\}}\left(|p_s(g)||\hat \eta (g)|+p_s(g)^2v_\eta(g)\right)\mu(dg)\right)ds & < & +\infty\;,\quad \P-a.s.
\enq
and we
define the set 
\beqs
\Ic(G) & = & \Big\{\text{Processes }(\pi_t=p_t(\Delta G_t))_{t\in[0,T]} \mbox{ with } p  \text{ being }\mathcal{P} \otimes \Bc(\R)\text{-measurable  satisfying } \eqref{cond-loc-smg} \Big\}.
\enqs
The extension of the stochastic integral w.r.t. $\tilde N$ of $\Pc\vee \sigma(G)$-measurable processes is stated as follows. 
\begin{Theorem}\label{extentionISN} There exists a unique continuous linear map $I$ from $\Ic(G)$ to the set of $\F$-adapted c\`adl\`ag processes such that 
\beqs
I_t(\pi) & = & \int_0^t \int_{\gamma\neq0}\pi_s\eta(e)\tilde N(de,ds)\;,\quad t\in[0,T] \;,
\enqs
for any $\pi\in \Ic(G)$. Continuity is understood in the following sense
\beqs
I(\pi^n) & \xrightarrow[n\rightarrow+\infty]{} & I(\pi)
\enqs uniformly in probability for any $\pi=p(\Delta G)\in \Ic(G)$ and  $\pi^n=p^n(\Delta G)\in \Ic(G)$, $n\geq 1$, such that
\beqs
\int_0^T\left((p_s-p^n_s)^2(0)+\int_{\gamma(\R)\setminus\{0\}}\hspace{-4mm}\left(|(p_s-p^n_s)(g)||\hat \eta (g)|+(p_s-p^n_s)^2(g)v_\eta(g)\right)\mu(dg)\right)ds & \xrightarrow[n\rightarrow+\infty]{\P} & 0.
\enqs
For $\pi=p(\Delta G)\in \Ic(G)$, the process $I(\pi)$ is a special semimartingale admitting the following Doob-Meyer decomposition 
\beq%gin{equation}
\label{decompSMIS}
 I_t(\pi) & = & M^1_t(p)+M^2_t(p)+A_t(p) 
\enq
where $M^1(p)$ and $M^2(p)$ are local martingales given by
\beqs
    M^{1}_t(p) & = & \int_0^t\int_{\gamma\neq0}p_s(\gamma(e))\hat \eta (\gamma(e))\tilde N(ds,de)\\
     M^{2}_t(p) & = & \int_0^t\int_{\gamma\neq0}p_s(\gamma(e))\big(\eta (e)-\hat \eta (\gamma(e))\big) N(ds,de)
\enqs
for $t\in[0,T]$ and $A(p)$ is a finite variation process given by
\beqs
    A_t(p) & = & \int_0^t\int_{\gamma\neq0} (p_s(g)-p_s(0))\hat \eta (g)\mu(dg)ds
\enqs
for $t\in[0,T]$.
If in addition the expectation of the integral appearing in \eqref{cond-loc-smg} is finite, then $M^1$ is a martingale with finite
expected total variation over $[0,T]$, $M^2$ is a square-integrable martingale with $L^2$-norm
\beqs
 \E\big[|M^2_T|^2\big]  & = & \E\Big[ \int_0^T
 \int_{\gamma(\R)\setminus\{0\}}|p_s(g)|^2v_\eta(g)\mu(dg)ds\Big]\;,
\enqs and 
$A$ has finite expected total variation
over $[0,T]$.
\end{Theorem}
In the sequel, we still denote the extended stochastic integral $I_t(\pi)$ by
%\beqs
% I_t(\pi) & = & 
$\int_0^t\int_{\gamma\neq 0} \pi_s\eta(e)\tilde N(ds,de)
$ % \enqs
for all $t\in[0,T]$ and $\pi\in\Ic(G)$.

 We notice that the only difference between Theorem \ref{extentionISN} and \cite[Theorem 2.4]{bank2022merton} is that we consider a stochastic integral over all possible jumps whereas  \cite[Theorem 2.4]{bank2022merton} considers only jumps with size smaller than one. However, the same arguments can be applied for the proof which is therefore omitted.

Finally, we shall consider the set of admissible strategies with signal $\Ac_{\mathrm{sgn}}$ defined by
\beqs
\Ac_{\mathrm{sgn}} & = & \Big\{\text{Processes }(\pi_t=p_t(\Delta G_t))_{t\in[0,T]} \text{ satisfying  } \eqref{cond-Adm-pi}-\eqref{cond-loc-smg} \mbox{ with } p  \text{ being }\mathcal{P} \otimes \Bc(\R)\text{-measurable}\Big\} \;.
\enqs
For $\pi\in \Ac_{\mathrm{sgn}}$ and for an initial endowment $x\in\R$, the self financing wealth process $ X^{x,\pi}$ defined by $ X_0^{x,\pi}=x$ and 
\begin{equation*}
           d X_{t}^{x, \pi}  %= \pi_{t}[\kappa dt + \sigma dW_t + \int_{\R}\eta(e) \tilde{N}(dt, de)] \\
           = {\pi}_t[\kappa dt + \sigma dW_t]  + \int_{\R} {\pi}_t \eta(e) \tilde{N}(dt, de),\quad  t\in [0, T],
\end{equation*}
is then well posed by Theorem \ref{extentionISN}.

\section{Exponential utility maximization}\label{sect4}
\subsection{The utility maximization problem}

We use the notation $u$ both for the utility function and as a dummy variable replacing the usual $U$ term in BSDEJs, for example in~\eqref{defDriverfexp}. This dual use may lead to ambiguity, but these are standard notations and the intended meaning should be clear from the context.

We %fix a utility function $u:~\R\rightarrow\R$ we
 consider an exponential utility function $u:~\R\rightarrow\R$ given by 
\begin{equation}
    u(x) = -\exp(-\lambda x)\;,\quad x\in\R\;,
\end{equation}
where $\lambda$ is a given positive constant which quantify the risk aversion of the investor. 

We also consider an $\Fc_T$-measurable random variable $F$ representing the payoff of some financial product. We look for the optimal expected utility for an investor selling the product $F$. This remains to maximize $\E\big[u(X^{x,\pi}_T - F)\big]$ over the set of allowed strategies. 

We then denote by $\bar \Ac_{\mathrm{sgn}}$  the subset of $\Ac_{\mathrm{sgn}}$ (resp.  $\bar \Ac$  the subset of $\Ac$) representing the set of allowed strategies with signal (resp. without signal). To define this set, we shall also impose an additional condition %of limited credit line 
on $\pi$. More precisely, we fix some constants $\overline \pi, \underline \pi>0$ and we ask $\pi$ to satisfy
\beq
\label{limited-credit-line}
\pi_t & \in & [-\underline \pi,\overline \pi] \quad \mbox{ for all }t\in[0,T]\;,~\P-a.s.
\enq%\textcolor{red}{put this condition later}
We observe that under \reff{limited-credit-line}, conditions \eqref{cond-Adm-pi} and \eqref{cond-loc-smg} are satisfied. Therefore, the sets  $\bar \Ac_{\mathrm{sgn}}$ and $\bar \Ac$  are therefore given by
\beqs
    \bar \Ac_{\mathrm{sgn}} & = & \Big\{ \text{Processes }\pi=(p_t(\Delta G_t)_{t\in[0,T]} %\text{ that are }
    ~\mathcal{P} \vee \sigma(G)\text{-measurable and satisfy  } %\eqref{cond-Adm-pi}-\eqref{cond-loc-smg}-
    \eqref{limited-credit-line}\Big\}\;,\\
     \bar \Ac & = & \Big\{ \text{Processes }\pi
    ~~\mathcal{P} \text{-measurable and satisfy  } %\eqref{cond-Adm-pi}-
    \eqref{limited-credit-line}\Big\}\;.
\enqs
We then define %the subset $\bar \Ac_{\mathrm{sgn}}$ of $\Ac_{\mathrm{sgn}}$ (resp.  $\bar \Ac$ of $\Ac$) representing the set of allowed strategies. %This set $\bar \Ac_{\mathrm{sgn}}$ (resp. $\bar \Ac$) may depend on the utility function we shall consider.
%
%Our goal is to compute 
the value functions $V$ and $V_{\mathrm{sgn}}$ by
\beq\label{defVw(x)}
    V(x) & = & \sup_{\pi \in \bar\Ac}\E\big[u(X^{x,\pi}_T - F)\big]\;,
\enq 
and
\beq\label{defV(x)}
    V_{\mathrm{sgn}}(x) & = & \sup_{\pi \in \bar\Ac_{\mathrm{sgn}}}\E\big[u(X^{x,\pi}_T - F)\big]\;,
\enq
for $x\in\R$. We refer to \cite{MAM08} for the study of the utility maximization problem without signal  \reff{defVw(x)} and we shall focus on the study of the utility maximization problem with signal \reff{defV(x)}.

For a given initial endowment $x\in\R$, we shall also look for  a related optimal strategy $\pi^*\in \bar \Ac_{\mathrm{sgn}}$, that is a strategy $\pi^*\in \bar \Ac_{\mathrm{sgn}}$ such that 
\begin{equation}\label{defOI}
    V_{\mathrm{sgn}}(x) = \E\big[u(X^{x,\pi^*}_T - F)\big].
\end{equation}
\subsection{Martingale optimality principle}
To solve problem \eqref{defV(x)}-\eqref{defOI}, we follow the martingale optimality approach presented in \cite{hu2005utility}. This approach consists in the characterization of the optimality by a martingale criterion as shown by the following result.
\begin{Proposition}\label{MgOptPr}
    Suppose that there exists a family of processes $(R^\pi)_{\pi\in \bar \Ac_{\mathrm{sgn}}}$ satisfying the following conditions. 
\begin{itemize}
\item[\textnormal{(i)}] $R_T^\pi= u(X^{x,\pi}_T-F)$ for any $\pi\in \bar \Ac_{\mathrm{sgn}}$.
\item[\textnormal{(ii)}] There is some constant $R_0$ such that $R^\pi_0=R_0$ for any $\pi\in\bar \Ac_{\mathrm{sgn}}$.
\item[\textnormal{(iii)}] The process $(R^\pi_t)_{t\in[0,T]}$ is a supermartingale for any $\pi\in \bar \Ac_{\mathrm{sgn}}$.
\item[\textnormal{(iv)}] There is some $\pi^*\in\bar \Ac_{\mathrm{sgn}}$ such that $(R^{\pi^*}_t)_{t\in[0,T]}$ is a martingale.
\end{itemize}
Then $R_0=V_{\mathrm{sgn}}(x)$ and $\pi^*$ is an optimal investment strategy.
\end{Proposition}
\begin{proof}
    Fix some $\pi\in \bar \Ac_{\mathrm{sgn}}$. We then have from (i), (ii) and (iii)
    \beqs
\E[u(X^{x,\pi}_T-F)] & = & \E[R^{\pi}_T]~\leq~R^{\pi}_0~=~R_0\;.
    \enqs
    Therefore, we get
    \beqs
    \sup_{\pi\in\bar \Ac_{\mathrm{sgn}}}\E\big[u(X^{x,\pi}_T-F)\big]~\leq~R_0\;.
    \enqs
    Using (iv), we get
    \beqs
    \E\big[u(X^{x,\pi^*}_T-F)\big]~=~R_0\;.
    \enqs
    Hence $\pi^*$ is optimal and $V_{\mathrm{sgn}}(x)=R_0$\;.
\end{proof}
\section{Optimality and BSDEs with jumps}\label{sect5}
In the sequel, we construct  a solution to problem \eqref{defV(x)}-\eqref{defOI} based on a family $(R^\pi)_{\pi\in \bar \Ac_{\mathrm{sgn}}}$ satisfying Proposition \ref{MgOptPr}. 

For this purpose, we use the theory of Backward SDEs with jumps (BSDEJ) for short. We next present the notion of BSDEJ and we refer to \cite{delong13} for a detailed presentation of this theory.  
We introduce the following spaces of processes. 
%\vspace{10mm}

%We first introduce some notations:

\begin{itemize}
    \item \( S^{\infty} = \left\{ \F\text{-adapted c\`adl\`ag processes } Y  \text{ valued in $\R$ such that } \esssup_{t \in [0,T]} |Y_t| < \infty \right\} \).

    \item \( S^{2} = \left\{ \F\text{-adapted c\`adl\`ag processes } Y  \text{ valued in $\R$ such that } \E ( \sup_{t \in [0,T]} |Y_t|^2 ) < \infty \right\} \).
    
    \item \( L^2(W) = \left\{\Pc\text{-measurable processes } Z \text{ valued in $\R$ such that } \mathbb{E} \left[\int_0^T |Z_t|^2 \, dt \right] < \infty \right\} \).
    
    \item \( L^2(\tilde{N}) = \left\{\mathcal{P} \otimes \mathcal{B}(\R)\text{-measurable processes } U \text{ valued in $\R$ such that } \mathbb{E} \left[\int_0^T \int_\R U_t^2(e) \, \nu(de) \, dt \right] < \infty \right\} \).

 \item \( L^\infty(\tilde{N}) = \big\{\mathcal{P} \otimes \mathcal{B}(\R)\text{-measurable processes } U \text{ valued in $\R$ such that } \P\otimes\nu( |U|> C)=0\mbox{ for some } C\in\R \big\} \).
\end{itemize}

We also define the following sets of functions.
\begin{itemize}
    \item $L^2(\nu) = \{ u : \R \xrightarrow{} \mathbb{R} \mbox{ Borel, } \int_{\R}u^2(e)\nu(de) <\infty \}$. 
    \item $L^{\infty}(\nu)=\{ u : \R \xrightarrow{} \mathbb{R} \mbox{ Borel, } \nu(|u|>C) =0 \mbox{  for some } C\in\R \}$. %as all the Borel function $u:\R\xrightarrow{}\mathbb{R}$ with bounded values $\nu$-a.e.  and equipped with the topology of convergence in measure. 
\end{itemize}
We next fix a terminal condition given an $\Fc_T$-measurable random variable $\xi$ and a function $f:~\Omega\times [0,T]\times \R\times \R^d\times L^2(\nu)$ that is assumed to be $\Pc\otimes\Bc(\R)\otimes\Bc(\R^d)\otimes\Bc(L^2(\nu))$-measurable.

A solution to the BSDEJ with parameters $(\xi,f)$ is a triple of processes $(Y, Z, U)\in S^\infty \times L^2(W) \times L^2(\tilde{N})$ satisfying 
\begin{equation}
Y_t  =  \xi + \int_{t}^{T} f(s, Y_s, Z_s, U_s) \, ds 
- \int_{t}^{T} Z_s \, dW_s 
- \int_{t}^{T} \int_{\R} U_s(e) \, \tilde{N}(ds,de), \quad t \in [0, T]\;.\label{BSDEJ}
 %& = & 
\end{equation}
%which is characterized by a bounded terminal condition  $F$ and a generator $f$ satisfying $$\int_{0}^{T}|f(s, Y_s, Z_s, U_s)|ds < \infty, \mathbb{P}-\mbox{as},$$is a triple of processes $(Y, Z, U)$ which is in $S^\infty \times L^2(W) \times L^2(\bar{N})$.

Following the approach initiated by \cite{hu2005utility} in the Brownian framework and then developed by \cite{MAM08} in the mixed Brownian-Poisson case, we look for the family $(R^\pi_t)_{t\in[0,T]}^{\pi\in \bar\Ac_{\mathrm{sgn}}}$ satisfying Proposition \ref{MgOptPr} under the following form
\beqs
R^\pi_t & = & u\big(X^{\pi}_t-Y_t\big)\;,\quad t\in[0,T]\;,
\enqs
where $(Y_t)_{t\in[0,T]}$ satisfies a BSDE with some parameters $(\xi,f)$. We next face two questions. The first one consists in finding appropriate $\xi$ and $f$ to satisfy conditions (i) to (iv) of Proposition \ref{MgOptPr}. Given, this appropriate coefficient, the second question is to prove existence, and possibly uniqueness, of solution to the considered BSDEJs. 

%\section{Exponential utility case}
%Right now, considering the utility maximization problem. More precisely, given a utility function $u:\mathbb{R}\xrightarrow{} \mathbb{R}$ and $F$ a contingent claim, that is a random payoff at time $T$ supposed to be $\mathcal{F}_T$-measurable and bounded, we define $$V(x) = \sup_{\pi\in A} \mathbb{E}[u(X_{T}^{x,\pi} - F)]$$
%the maximal expected utility we can achieve by starting at time $0$ with initial capital $x$ using some strategy $\pi \in A$. \\
\subsection{Related BSDEJ}
%In this section, we consider an exponential utility function $u$ given by 
%\begin{equation}
%    u(x) = -\exp(-\lambda x)\;,\quad x\in\R\;,
%\end{equation}
%where $\lambda$ is a given positive constant which %quantify the risk aversion of the investor.

\vspace{2mm}

For $u \in L^2(\nu)$ and $\lambda > 0$, we define the following non-negative functional:
\beqs
|u|_{\lambda} := \int_\R h_{\lambda}(u(e))\, \nu(de),
\enqs where $h_\lambda$ is the function defined by $$h_\lambda(x) = \frac{e^{\lambda x} - \lambda x - 1}{\lambda}$$ for all $x\in\R$.

%\vspace{2mm}
%We consider the case where the 'estimation' of the jumps of the asset given the signal is negative, that is
%\beqs
%\hat \eta (g) & < & 0
%\enqs
%for $\mu$-a.a. $g\in\gamma(\R)\setminus\{0\}$.

%In this case the optimal strategy might be to have a negative investment on the risky asset $S$. Therefore, 
%We shall also impose an additional condition %of limited credit line 
%on $\pi$. More precisely, we fix some constants $\overline \pi, \underline \pi>0$ and we ask $\pi$ to satisfy
%\begin{equation}
%\label{limited-credit-line}
%\pi_t  \in  [-\underline \pi,\overline \pi] \quad \mbox{ for all }t\in[0,T]\;,~\P-a.s.
%\end{equation}%\textcolor{red}{put this condition later}
%In this section, our set of admissible strategies $\bar %\Ac_{\mathrm{sgn}}$ is therefore given by
%\begin{equation}
%    \bar \Ac_{\mathrm{sgn}} = \Big\{ \text{Processes }\pi=(p_t(\Delta G_t)_{t\in[0,T]} %\text{ that are }
%    ~\mathcal{P} \vee \sigma(G)\text{-measurable and satisfy  } \eqref{cond-Adm-pi}-\eqref{cond-loc-smg}-\eqref{limited-credit-line}\Big\}\;.
%\end{equation}
%To characterize the value function $V(x)$ and the optimal strategy, we construct as in [Hu, Imkeller, Muller] a family of stochastic processes $R^{(\pi)}$ such that:

%\begin{enumerate}
%    \item \( R^{(\pi)}_{T} = -\exp(-\lambda(X^{x,\pi}_{T} - F)) \) for all \( \pi \in A \).
%    \item \( R^{(\pi)}_{0} = R_0 \) is constant for all \( \pi \in A \).
%    \item \( R^{(\pi)} \) is a supermartingale for all \( \pi \in A \), and there exists a \( \hat{\pi} \in A \) such that \( R^{\hat{\pi}} \) is a martingale.
%\end{enumerate}

Following the approach described in the previous section, we construct the family of processes $(R^\pi)_{\pi\in \bar \Ac_{\mathrm{sgn}}}$ introduced in Proposition \ref{MgOptPr}, by setting
\begin{equation}\label{defRpi}
R^{\pi}_{t} = -\exp\left(-\lambda\left(X^{x,\pi}_{t} - Y_t\right)\right), \quad t \in [0, T], \; \pi \in \bar \Ac_{\mathrm{sgn}},
\end{equation}
where $(Y, Z, U)$ is a solution to the BSDEJ \eqref{BSDEJ}. From condition (i) of Proposition \ref{MgOptPr}, we need to have $\xi=F$. Therefore $(Y, Z, U)$ satisfies 
\begin{equation}\label{BSDEexp}
Y_t = F + \int_{t}^{T} f(s, Y_s, Z_s, U_s) \, ds - \int_{t}^{T} Z_s \, dW_s - \int_{t}^{T} \int_{\R} U_s(e) \, \tilde{N}(ds, de), \quad t \in [0, T].   
\end{equation}

%To precise the definition of a solution to equation we introduce the following spaces.

%\begin{itemize}
   % \item \( S^{\infty} = \left\{ \F\text{-adapted c\`adl\`ag processes } Y  \text{ valued in $\R$ such that } \esssup_{t \in [0,T]} |Y_t| < \infty \right\} \).
    
    %\item \( L^2(W) = \left\{\Pc\text{-measurable processes } Z \text{ valued in $\R$ such that } \mathbb{E} \left[\int_0^T |Z_t|^2 \, dt \right] < \infty \right\} \).
    
    %\item \( L^2(\tilde{N}) = \left\{\mathcal{P} \otimes \mathcal{B}(E)\text{-measurable processes } U \text{ valued in $\R$ such that } \mathbb{E} \left[\int_0^T \int_\R U_t^2(e) \, \nu(de) \, dt \right] < \infty \right\} \).

    %\item The set $L^2(\nu) = \{ u : E \xrightarrow{} \mathbb{R}, \int_{\R}u^2(e)\nu(de) <\infty \}$. Moreover, we define $L^{\infty}(\nu)$ as all the function $u:E\xrightarrow{}\mathbb{R}$ with bounded values (almost surely) and equipped with the topology of convergence in measure. 
%\end{itemize}
%For $u \in L^2(\nu)$ and $\lambda > 0$, we define the following non-negative functional:
%\beqs
%|u|_{\lambda} := \int_\R g_{\lambda}(u(e))\, \nu(de),
%\enqs where $g_\lambda(x) = \frac{e^{\lambda x} - \lambda x - 1}{\lambda}.$

%In this framework we look for a  driver $f$ such that $R^{(\pi)}$ is a supermartingale for all $\pi \in A$ and there is $\hat{\pi}$ for which $R^{(\hat{\pi})}$ is a martingale.\\
In order to compute $f$, we apply Ito formula to $R^{\pi}$. Define the process $L$ by $L_t = X^{x, \pi}_t - Y_t$ for $t\in[0,T]$. Using \eqref{decompSMIS} and since $\pi_t$ coincides with $p_t(0)$ for almost all $t\in[0,T]$, we have 
\beqs
dL_t &= &  \left[ p_t(0)\kappa + f(t, Y_t, Z_t, U_t) 
+ \int_{\gamma(\R) \setminus \{0\}} \left(p_t(g) - p_t(0)\right) \hat{\eta}(g) \mu(dg) \right] dt \\
& & + \left[p_t(0)\sigma - Z_t\right] dW_t 
\\
& &  + \int_{\R} \Big(p_t(\gamma(e))\eta(e) - U_t(e) \Big)\tilde{N}(dt, de) 
\;,\quad t\in[0,T]\;.
\enqs
Applying  Itô's formula to $R^\pi=u(L)$ gives
\beqs
dR_t^\pi &= & R_{t-}^\pi \Bigg[ \Bigg(-\lambda \Big( p_t(0)\kappa + f(t, Y_t, Z_t, U_t) 
+ \int_{\gamma(\R) \setminus \{0\}} \left(p_t(g) - p_t(0)\right) \hat{\eta}(g) \mu(dg) \\
& & \qquad \qquad  - \int_\R \Big( 
p_t(\gamma(e)) \eta(e) - 
 U_t(e) 
%- p_t(0)\eta(e)\mathds{1}_{\{\gamma = 0, |\eta|\leq 1\}} 
%\\
%&
%\quad - p_t(\gamma(e)) \hat{\eta}(\gamma(\R)) \mathds{1}_{\{\gamma \neq 0, |\eta|\leq 1\}}
\Big) \nu(de) \Big) 
+ \frac{1}{2} \lambda^2 \left(p_t(0)\sigma - Z_t\right)^2 \Bigg) dt \\
& & \qquad \qquad  - \lambda \left(p_t(0)\sigma - Z_t\right) dW_t 
+ \int_\R \left( \exp\left(-\lambda(p_t(\gamma(e))\eta(e) - U_t(e))\right) - 1 \right) N(dt, de)\Bigg]\;,\quad t\in[0,T]\;. 
\enqs
Therefore we get
\beqs
dR_t^\pi &= & R_{t-}^\pi \Big[ \Lambda_t dt +\Gamma_tdW_t+\int_\R \Theta_t(e)\tilde N(de,dt)\Big]
\enqs
with
\beqs
\Gamma_t & = & - \lambda \left(p_t(0)\sigma - Z_t\right) \\
\Theta_t(e) & = &  \exp\left(-\lambda(p_t(\gamma(e))\eta(e) - U_t(e))\right) - 1 \\
\Lambda_t & = & -\lambda \Big( p_t(0) \kappa + f(t, Y_t, Z_t, U_t) 
+ \int_{\gamma(\R) \setminus \{0\}} \left(p_t(g) - p_t(0)\right) \hat{\eta}(g) \mu(dg) \\
& & \qquad \qquad  - \int_\R \Big( \frac{1}{\lambda}\Theta_t(e)+
p_t(\gamma(e)) \eta(e) - 
 U_t(e) \Big) \nu(de) \Big) + \frac{1}{2} \lambda^2 \left(p_t(0)\sigma - Z_t\right)^2  
\enqs
for $t\in[0,T]$ and $e\in \R  $.
%
%Therefore, we get
%\begin{align*}
%dR^\pi_{t-}  = &  R_{t-}^\pi \Bigg[\Bigg( -\lambda \Bigg( p_t(0)\big(\kappa-\int_{|\eta|\leq1}\eta(e)\nu(de)dt\big) + f(t, Z_t, U_t) 
%+ \int_{\R} U_t(e)\nu(de)\Bigg) %\\
%&%\quad + \int_\R \Big( 
%p_t(\gamma(e)) \eta(e) \mathds{1}_{|\eta| > 1} 
% U_t(e) 
%- p_t(0)\eta(e)\mathds{1}_{\{\gamma = 0, |\eta|\leq 1\}} \\
%&\quad - p_t(\gamma(e)) \hat{\eta}(\gamma(\R)) \mathds{1}_{\{\gamma \neq 0, |\eta|\leq 1\}} \Big) \nu(de) \Bigg) 
%+ \frac{1}{2} \lambda^2 \left(p_t(0)\sigma - Z_t\right)^2 \Bigg) dt \\
%& - \lambda \left(p_t(0)\sigma - Z_t\right) dW_t 
%+ \int_\R \left( \exp\left(-\lambda(p_t(\gamma(e))\eta(e) - U_t(e))\right) - 1 \right) N(dt, de)\Bigg]\;,\quad t\in[0,T]\;.
%\end{align*}

To satisfy condition (iii) of Proposition \ref{MgOptPr}, we impose the  inequality $\Lambda_t\geq 0$ which gives
\beqs
 -\lambda f(t, Y_t, Z_t, U_t) 
- \lambda \left( p_t(0)\big(\kappa-\int_{\R}\eta(e)\nu(de)\big)   + \int_{\R}  U_t(e)  \nu(de) \right)&  \\
+ \frac{1}{2} \lambda^2 \left(p_t(0)\sigma - Z_t\right)^2 %&\\
% + \big|p_t(\gamma)-U_t\big|_\lambda
 +\int_{\R} \left( \exp\left(-\lambda \left(p_t(\gamma(e)) \eta(e) - U_t(e)\right) \right) - 1 \right) \nu(de)
 & \geq  & 0.
\enqs
%\begin{align*}
% -\lambda f(t, Z_t, U_t) 
%- \lambda \left( p_t(0)\kappa + \int_{\gamma(\R) \setminus \{0\}} \left(p_t(g) - 2p_t(0)\right) K(g, |\eta| \leq 1)\hat{\eta}(g)\mu(dg) \right.& \\
% \left. \qquad - \int_{\R} \left[ p_t(\gamma(e)) \eta(e)\mathds{1}_{\gamma\neq0,\;|\eta|\leq 1}(e) - U_t(e) \right] \nu(de) \right) 
%+ \frac{1}{2} \lambda^2 \left(p_t(0)\sigma - Z_t\right)^2 &\\
% + \int_{\R} \left( \exp\left(-\lambda \left(p_t(\gamma(e)) \eta(e) - U_t(e)\right) \right) - 1 \right) \nu(de) & \geq 0.
%\end{align*}
Let us define the constant
\begin{equation}
C_{\kappa, \eta} = \frac{1}{\lambda\sigma}\Big(
%\kappa - 2\int_{\gamma(\R) \setminus \{0\}} K(g, |\eta| \leq 1)\hat{\eta}(g)\mu(dg)
\kappa-\int_{\R}\eta(e)\nu(de)\Big).    
\end{equation}
Then we have
\beqs
 -\lambda f(t,Y_t,  Z_t, U_t) 
+ \frac{1}{2} \lambda^2 \left[p_t(0)\sigma - \left(Z_t + \frac{C_{\kappa, \eta}}{\lambda} \right) \right]^2 
- \lambda Z_t C_{\kappa, \eta} - \frac{C_{\kappa, \eta}^2}{2} &   & \\
 \quad + \int_{\R} \left[ \exp\left(\lambda \left(U_t(e) - p_t(\gamma(e)) \eta(e) \right) \right) - 1 - \lambda %\left(
 U_t(e) 
 %- p_t(\gamma(e)) \eta(e)\mathds{1}_{\gamma\neq0,\;|\eta|\leq 1}(e) \right) 
 \right] 
 \nu(de)  % \\
% \quad -\lambda \int_{\gamma(\R) \setminus \{0\}} p_t(g) K(g, |\eta| \leq 1)\hat{\eta}(g)\mu(dg)  & 
 & \geq &  0\;.
\enqs
Using the disintegration formulation on the set $\{\gamma \neq 0\}$, we can rewrite the previous inequality as
\begin{align*}
 -\lambda f(t,Y_t,  Z_t, U_t) 
+ \frac{1}{2} \lambda^2 \left[p_t(0)\sigma - \left(Z_t + \frac{C_{\kappa, \eta}}{\lambda} \right) \right]^2 
- \lambda Z_t C_{\kappa, \eta} - \frac{C_{\kappa, \eta}^2}{2} &  \\
 \quad + \int_{\gamma=0} \left[ \exp\left(\lambda \left(U_t(e) - p_t(0) \eta(e) \right) \right) - 1 - \lambda 
 U_t(e) 
 \right] 
 \nu(de) &  \\
 \quad + \int_{\gamma(\R)\setminus\{0\}}\int_{\gamma=g} \left[ \exp\left(\lambda \left(U_t(e) - p_t(g) \eta(e) \right) \right) - 1 - \lambda 
 U_t(e) 
 \right] K(g,de)\mu(dg)
 %\nu(de)
 & 
~\geq~ 0.
\end{align*}
%\begin{align*}
%\int_{\gamma(\R) \setminus \{0\}} \int_{\gamma = g} 
%\left[ \exp\left(\lambda \left(U_t(e) - p_t(\gamma(e)) \eta(e) \right)\right) - 1 
%- \lambda \left(U_t(e) - p_t(\gamma(e)) \eta(e) \mathds{1}_{|\eta|\leq 1}(e)\right) \right] 
%K(g, de)\mu(dg).
%\end{align*}
% modifs de la bsde et mettre tout sous la meme integrale avec la mesure mu
% (utilisation de la formule de désintégration)
To satisfy the previous inequality, %condition (iii) of Proposition \ref{MgOptPr}, 
we define the driver $f$ by
%{\footnotesize
\begin{align}
f(z, u) =  \inf_{p \in[-\underline{\pi},\overline \pi]} f^1(z,u,p)
%\left\{ \frac{1}{2} \lambda \left[p - \left(z + \frac{C_{\kappa, \eta, K}}{\lambda} \right) \right]^2 + \int_{\gamma = 0} \left[ \frac{\exp(\lambda(u(e) - p \eta(e))) - 1}{\lambda} - (u(e) - p \eta(e)) \right] \nu(de) \right\} \\
\label{defDriverfexp}
 + \int_{\gamma(\R) \setminus \{0\}} \inf_{p \in[-\underline{\pi},\overline\pi]} f^2(g,u,p)
%\left\{ \int_{\gamma = g} \left[ \frac{\exp(\lambda(u(e) - p \eta(e))) - 1}{\lambda} - \left(u(e) - \mathds{1}_{\{\eta > 1\}} p \eta(e)\right) \right] K(g, de) \right\} 
\mu(dg) 
%\\
 - z C_{\kappa, \eta} - \frac{C_{\kappa, \eta}^2}{2 \lambda}
\end{align}
%}
with
\beqs
f^1(z, u,p) & =  & \frac{1}{2} \lambda \left[p \sigma  - \left(z + \frac{C_{\kappa, \eta}}{\lambda} \right) \right]^2 + \int_{\gamma = 0} \left[ \frac{\exp(\lambda(u(e) - p \eta(e))) - 1}{\lambda} - u(e)
%- p \eta(e)\mathds{1}_{|\eta|\leq 1}(e))
\right] \nu(de) 
\enqs and
%for $z\in\R$, $u\in L^2(\nu)$ and $p\in\R$, and %\\%\label{defDriverf}
\beqs
f^2( g,u,p) & = & %\left\{
\int_{\gamma = g} \left[ \frac{\exp(\lambda(u(e) - p \eta(e))) - 1}{\lambda} - %\left(
u(e) %- \mathds{1}_{\{|\eta| > 1\}} p \eta(e)\right) 
\right] K(g, de) %\right\} %\mu(dg) %\\
% - z C_{\kappa, \eta, K} - \frac{C_{\kappa, \eta, K}^2}{2 \lambda}.
\enqs
for $z\in\R$, $u\in L^2(\nu)\cap L^\infty(\nu)$, $g\in\gamma(\R)$ and $p\in\R$.
We notice that $f$ does not depend neither on the time variable $t$, nor on the component $Y$. 
We finally get the following result. 
\begin{Theorem}
    Suppose that %$\nu(\eta>0)=0$ 
    %$\hat \eta(g)<0$ for $\mu$-a.a. $g\in\gamma(\R)\setminus \{0\}$. 
   %Suppose also that 
   the Backward SDE \eqref{BSDEJ} with parameters $\xi=F$ and $f$ given by \eqref{defDriverfexp} admits a solution $(Y,Z,U)\in S^\infty \times L^2(W) \times L^2(\tilde{N})\cap L^\infty(\tilde{N})$. Then, there exists a Borel map %\textcolor{red}{need $\gamma(\R)$ open }
   $p^*:~\R\times\R\times  L^2(\nu)\cap L^\infty(\nu)\rightarrow\R$  satisfying
    \beq\label{mindriveratpistar}
f(z, u) & = & f^1(z,u,p^*(0,z,u))+
%\frac{1}{2} \lambda \left[p^*(0,z,u) - \left(z + \frac{C_{\kappa, \eta, K}}{\lambda} \right) \right]^2 \\
% & + \int_{\gamma = 0} 
% \left( \frac{\exp(\lambda(u(e) - p^*(e,z,u) \eta(e))) - 1}{\lambda} - (u(e) - p^*(e,z,u) \eta(e)) \right) \nu(de)  \\\label{Driverfp}
%&+ 
\int_{\gamma(\R) \setminus \{0\}}  f^2(g,u,p^*(g,z,u)) %\int_{\gamma = g} \left( \frac{\exp(\lambda(u(e) - p(e,z,u) \eta(e))) - 1}{\lambda} - \left(u(e) - \mathds{1}_{\{\eta > 1\}} p(e,z,u) \eta(e)\right)  \right)K(g, de) 
\mu(dg) %\\
%&
- z C_{\kappa, \eta} - \frac{C_{\kappa, \eta}^2}{2 \lambda}
\enq
and
    \beqs
        p^*(g,z,u) & \in &  [-\underline{\pi},\overline \pi] 
    \enqs
for all $g,z,u\in\gamma(\R)\times\R\times L^2(\nu)\cap L^\infty(\nu)$, and we have $V(x)=u(x-Y_0)$ and $\pi^*=(p^*(\Delta G_t,Z_t,U_t))_{t\in[0,T]}$ is an optimal strategy in $\bar \Ac_{\mathrm{sgn}}$.
\end{Theorem}
\begin{proof} 
Let $(Y,Z,U)\in S^\infty \times L^2(W) \times L^2(\tilde{N})$  be solution to BSDE \eqref{BSDEexp} with $f$ given by \eqref{defDriverfexp} and consider the family $(R^\pi)_\pi$ defined by \eqref{defRpi}. We prove that the conditions of Proposition \ref{MgOptPr} are satisfied.

The family $(R^\pi)_\pi$  satisfies (i) and (ii) of Proposition \ref{MgOptPr} as $R^\pi_T=u(X_T^{x,\pi}-Y_T)=u(X_T^{x,\pi}-F)$ and $R_0^\pi=u(X_0^{x,\pi}-Y_0)=u(x-Y_0)$ for any $\pi\in \bar \Ac_{\mathrm{sgn}}$.  From the definition of $f$ and the computations made above, we observe that $R^\pi$  is a local supermartingale. We show that $R^\pi$ is even a supermartingale since $Y \in S^{\infty} $ and the process $\tilde R^\pi = u(X^\pi)$ is uniformly integrable. Indeed, 
$\displaystyle \sup_{t\in [0,T]}\E (|\tilde R_t^\pi|^2) < +\infty$ because
\beqs
d\tilde R_t^\pi &= & \tilde R_{t-}^\pi \Big[ \tilde\Lambda_t dt +\tilde\Gamma_tdW_t+\int_\R \tilde\Theta_t(e)\tilde N(de,dt)\Big],
\enqs
with $\tilde\Gamma$, $\tilde \Theta(e)$ and $\tilde \Lambda$ uniformly bounded since
\beqs
\tilde\Gamma_t & = & - \lambda p_t(0)\sigma \\
\tilde \Theta_t(e) & = &  \exp \left(-\lambda p_t (\gamma(e) )\eta(e) \right) - 1 \\
\tilde \Lambda_t & = & -\lambda \Big( p_t(0) \kappa
+ \int_{\gamma(\R) \setminus \{0\}} \left(p_t(g) - p_t(0)\right) \hat{\eta}(g) \mu(dg)\\
& & \qquad \qquad - \int_\R \Big( \frac{1}{\lambda} \tilde \Theta_t(e)+
p_t(\gamma(e)) \eta(e)  \Big) \nu(de) \Big) + \frac{1}{2} \left(\lambda p_t(0)\sigma\right)^2 
\enqs
For $\tilde \Lambda$, a Taylor expansion is needed to manage the integrated term with respect to $\nu(de)$. 

Now, suppose there exists a Borel function $p^*$ satisfying \reff{mindriveratpistar}. We show that  $\pi^*\in\bar \Ac_{\mathrm{sgn}}$ and that condition (iv) is satisfied with $\pi^*$. Since $p_t^*(g)\in[-\underline \pi,\overline \pi]$ for all $t\in[0,T]$ and $g\in \Gamma(\R)$, we get from \reff{condAdmContSignal1} and \reff{condAdmContSignal2}
\beqs
\int_0^T%\hspace{-2mm}
\Big(p^*_s(0)^2+\int_{\gamma(\R)\setminus\{0\}}%\hspace{-10mm}
\left(|p^*_s(g)|\hat \eta (g)|+p^*_s(g)^2v_\eta(g)\right)\mu(dg)\Big)ds & \leq & \\
T (\underline \pi \wedge \overline \pi)^2\left(1+\int_{\gamma(\R)\setminus\{0\}} v_\eta(g) \mu(dg)\right)+T(\underline \pi \wedge \overline \pi)\int_{\gamma(\R)\setminus\{0\}}|\hat \eta (g)| \mu(dg) & < & +\infty\;.
\enqs

It remains to show that there exists a Borel function $p^*$ satisfying \reff{mindriveratpistar}. Using the dominated convergence theorem, we get that $f^1$ and $f^2$ are twice continuously differentiable w.r.t. $p$ on $\gamma(\R)\times \R\times L^2(\nu)\cap L^\infty(\nu)$ and we have
\beqs
\frac{\partial^2 f^1}{\partial p^2}(z,u,p) & = & \lambda\Big(\sigma   + \int_{\gamma = 0} \eta(e)^2 \exp(\lambda(u(e) - p \eta(e)))  \nu(de) \Big)\;,%~~ >~~0
\\
\frac{\partial^2 f^2}{\partial p^2}(g,u,p) & = & \int_{\gamma = g} \eta(e)^2 \exp(\lambda(u(e) - p \eta(e))) K(g, de)\;, %~~\geq~~0 
\enqs
for all $(g,z,u,p)\in \gamma(\R)\times\R\times L^2(\nu)\cap L^\infty(\nu)\times \R$.

This shows that  the functions $p\mapsto f^1(z,u,p)$ and $p\mapsto f^2(g,u,p)$ are strictly convex  for any $z\in\R$, $u\in L^2(\nu)\cap L^\infty(\nu)$ and $g\in \gamma(\R)$ such that $K(g,\eta\neq0)\neq 0$. If  $K(g,\eta\neq0)= 0$ then  $p\mapsto f^2(g,u,p)$ is constant. 
Therefore, a function $p^*$ satisfying \reff{mindriveratpistar} is given by
\beqs
\big\{(0, z,u,p^*(0,z,u))~:~(z,u)\in \R\times (L^2(\nu)\cap L^\infty(\nu))\big\} & = & A_1 \cup B_1\cup C_1
\enqs
where 
\beqs
A_1 & = & \Big\{(0, z,u,p)~:~ (z,u,p)\in \R\times \Oc_1 \times \R~\mbox{ and }~\frac{\partial f^1}{\partial p}(z,u,p)=0  \Big\}\\
%A_g & = & \Big\{(z,u,p)\in \R\times L^2(\nu)\cap L^\infty(\nu)\times \R~:~\frac{\partial f^1}{\partial p}(z,u,p)=0\mbox{ if }  \Big\}
B_1 & = & \Big\{(0,z,u,-\underline\pi )~:~(z,u)\in \R\times ((L^2(\nu)\cap L^\infty(\nu))\setminus\Oc_1)  ~~\mbox{ and }~~\frac{\partial f^1}{\partial p}(z,u,-\underline \pi)~\geq~0  \Big\} \\
C_1 & = & \Big\{(0,z,u,\overline\pi )~:~(z,u)\in \R\times ((L^2(\nu)\cap L^\infty(\nu))\setminus\Oc_1)  ~~ \mbox{ and }~~\frac{\partial f^1}{\partial p}(z,u,-\underline \pi)~<~0  \Big\}
\enqs
with
\beqs
\Oc_1 & = & \Big\{ (z,u)\in \R\times (L^2(\nu)\cap L^\infty(\nu)) \mbox{ such that } \frac{\partial^2 f^1}{\partial p^2}(z,u,-\underline \pi)\cdot \frac{\partial^2 f^1}{\partial p^2}(z,u,\overline \pi) ~<~0 \Big\}\;.
\enqs
and
\beqs
\big\{(g,z,u,p^*(g,z,u))~:~(g,z,u)\in (\gamma(\R)\setminus\{0\})\times\R\times (L^2(\nu)\cap L^\infty(\nu))\big\} & = & A_2 \cup B_2\cup C_2 \cup D_2
\enqs
where 
\beqs
A_2 & = & \Big\{(g,z,u,p)\in \Gamma\times \R\times \Oc_2 \times \R~:~\frac{\partial f^2}{\partial p}(g,u,p)=0  \Big\}\\
%A_g & = & \Big\{(z,u,p)\in \R\times L^2(\nu)\cap L^\infty(\nu)\times \R~:~\frac{\partial f^1}{\partial p}(z,u,p)=0\mbox{ if }  \Big\}
B_2 & = & \Big\{(g,z,u,-\underline\pi )~:~(g,z,u)\in \Gamma\times \R\times ((L^2(\nu)\cap L^\infty(\nu))\setminus\Oc_2)  ~~\mbox{ and }~~\frac{\partial f^2}{\partial p}(g,z,u,-\underline \pi)~\geq~0  \Big\} \\
C_2 & = & \Big\{(g,z,u,\overline\pi )~:~(g,z,u)\in \Gamma\times \R\times ((L^2(\nu)\cap L^\infty(\nu))\setminus\Oc_2)  ~~ \mbox{ and }~~\frac{\partial f^1}{\partial p}(g,u,-\underline \pi)~<~0  \Big\}\\
D_2 & = & \Big\{(g,z,u,\overline\pi )~:~(g,z,u)\in (\gamma(E)\setminus(\Gamma\cup\{0\}))\times \R\times ((L^2(\nu)\cap L^\infty(\nu))\setminus\Oc_2)   \Big\}
\enqs
with
\beqs
\Oc_2 & = & \Big\{ (g,z,u)\in (\gamma(\R)\setminus\{0\})\times\R\times (L^2(\nu)\cap L^\infty(\nu)) \mbox{ such that }  
%\\
% & & \qquad
 \frac{\partial^2 f^1}{\partial p^2}(z,u,-\underline \pi)\cdot \frac{\partial^2 f^1}{\partial p^2}(z,u,\overline \pi) ~<~0 %\mbox{ and } K(g,\eta\neq0)\neq 0
 \Big\}%\;.
\enqs
and 
\beqs
\Gamma & = & \Big\{ g\in \gamma(\R)\setminus\{0\} \mbox{ such that }   K(g,\eta\neq0)\neq 0\Big\}\;.
\enqs
We then notice that the sets $A_1,A_2, B_1,B_2,C_1,C_2$ and $D_2$ are Borel measurable since $\gamma(\R)$ is Borel measurable. Hence the  set 
$\big\{(g, z,u,p^*(0,z,u))~:~(g,z,u)\in \gamma(\R)\times\R\times (L^2(\nu)\cap L^\infty(\nu))\big\}$ is Borel measurable which ensures that the function $p^*$ is also Borel measurable.
\end{proof}
\subsection{Existence of solution to the BSDE}

%Throughout this section, we will assume, without loss of generality, that $\Phi(0)$ is compact. In the non-compact case, $\Phi(0)$ can be approximated by a sequence of compact sets, allowing us to apply the results established in the compact setting, as in~\cite{Morlais}.
We turn to the study of a solution to the
Backward SDE \eqref{BSDEJ} with parameters $\xi=F$ and $f$ given by \eqref{defDriverfexp}.

Following the approach of \cite{MAM08}, we introduce the sequence of finite measures $(\nu_m)_{m\geq1}$ defined by
\beqs
\nu_m(de) & = & \nu\big(de\cap \big(\R\setminus [\frac{-1}{m};\frac{1}{m}]\big)\big)
\enqs
and a sequence of 
function $(f_{m})_{m\geq 1}$ approximating the driver $f$ and defined by
\beqs
f_{m}(z,u) & = & \inf_{\overline \pi\geq p \geq -\underline{\pi}} f^{1}_m(z,u,p)
%\label{defDriverfm}
 + \int_{\gamma(\R) \setminus \{0\}} \inf_{\overline \pi\geq p \geq -\underline{\pi}} f^{2}_m(z,u,g,p)
\mu(dg) \\
& & 
 - z C_{\kappa, \eta} - \frac{C_{\kappa, \eta}^2}{2 \lambda}
\enqs
with
\beq
f^{1}_m(z, u,p) & = & \frac{1}{2} \lambda \left[p \sigma - \left(z + \frac{C_{\kappa, \eta}}{\lambda} \right) \right]^2\rho_m(z) \nonumber \\
 & & + \int_{\gamma = 0}  h_\lambda\Big(\varphi_m\big(u(e) - p \eta(e)))\Big) \nu_m(de)%\\ & & 
 - \int_{\gamma = 0}p\eta(e)\nu(de) \nonumber
\enq and
\beq
f^{2}_m(z, u,g,p) & = & 
\int_{\gamma = g} h_\lambda\Big(\varphi_m\big(u(e) - p \eta(e)\big)\Big)\rho_m(u(e)) \mathds{1}_{|e|>\frac{1}{m}}K(g, de)\\
 & & - \int_{\gamma = g}p \eta(e)K(g,de)\nonumber
\enq
for $z\in\R$, $u\in L^2(\nu)$, $g\in\gamma(\R)$, $p\in\R$ and $m\geq1$.
For $m\geq 1$, the  function $\rho_m$  is defined by
\beqs
\rho_m(x) & = & (x+m+1)\mathds{1}_{(-(m+1),-m)}(x)+ \mathds{1}_{[-m,m]}(x)+(m+1-x)\mathds{1}_{(m,m+1]}(x)
\enqs
for $x\in\R$. For $m\geq 1$, the function $\varphi_m$  is given by
\beqs
\varphi_m(x) & =  & \varphi(x-m)+m
\enqs
with 
\beqs
\varphi(x) & = & \arctan(x)\mathds{1}_{x>0}+x\mathds{1}_{x\leq 0}
%-\frac{1}{2}x^2+x
\enqs
for $x\in \R$.
%are assumed to satisfy the following conditions.
%\begin{itemize}
%    \item $\rho_m$ is be continuously differentiable for all $m\geq1$.
%    \item The sequence $(\rho_m(x))_m$ is nondecreasing for all $x\in\R$.
   % \item $\rho_m$ is increasing, that is $\rho_m'(x)>0$ for all $x\in\R$.
   % \item There exists a constant $C$ such that %$\rho_m(x)\in(-\infty,m+C]$ for $x\in\R$ and $m\geq1$.
   %\item $\rho_m(x)\in[0,1]$ for all $x \in \R$.
   % \item $\rho_m(x)=1$ for $|x| \leq m$. 
   % \item $\rho_m(x)=0$ for $|x| \geq m+1$.
    %\item $\rho_m(x)=m+\frac{1}{2}$ for $x\geq m+1$.
%\end{itemize} 
%An example of such a sequence $(\rho_m)_m$ is given by
%\beqs
%   \rho_m(x) & =  & \mathds{1}_{(-m,m]}(x)+\rho(x-m)\mathds{1}_{[m,m+1]}(x)
%    +\rho(x+m)\mathds{1}_{[-(m+1),-m]}(x)
%\enqs
%%\beqs
%%\rho_m(x) & =  & \rho(x-m)+m
%%\enqs
%with 
%\beqs
%\rho(x) & = & \exp\Big(1-\frac{1}{1-|x|^2}\Big)
%\enqs
%%\beqs
%%\rho(x) & = & \arctan(x)\mathds{1}_{x>0}+x\mathds{1}_{x\leq 0}
%%-\frac{1}{2}x^2+x
%%\enqs
%for $n,m\geq1$ and $x\in\R$. %and $C=\frac{\pi}{2}$.

We first have the following Lipschitz property for the approximated drivers.

\begin{Lemma}\label{LemLipfm} (i) The functions $f_m$, $m\geq 1$, are well defined on $\R\times L^2(\nu)$ and  the sequence $(f_m)_{m\geq 1}$ is nondecreasing:
\beqs
f_m(z,u) & \leq & f_{m+1}(z,u)
\enqs
for all $(z,u)\in\R \times L^2(\nu)$ and $m\geq 1$. 

\noindent (ii)    For $m\geq1$, the functions $f_{m}$ is Lipschitz continuous: there exists a positive constant $L_{m}$ such that
    \beqs
        |f_{m}(z,u)-f_{m}(z',u')| & \leq  & L_{m}\Big(|z-z'|+|u-u'|_{L^2(\nu)}\Big)
    \enqs
    for all $z,z'\in\R$ and $u,u'\in L^2(\nu)$.
\end{Lemma}
\begin{proof}
(i) To show the well posedness of $f_m$, it is sufficient to show that the integrals involving the component $u$ are well defined.
We first observe that from the definition of $\varphi_m$, the function $e\mapsto \varphi_m(u(e)-p\eta(e))$ is bounded. Therefore, there exists a constant $C$ such that
\beqs
0~~\leq~~h_\lambda\Big(\varphi_m\big(u(e)-p\eta(e)\big)\Big) & \leq & C|u(e)-p\eta(e)|^2 
\enqs
for all $e\in \R  $. Since $u,\eta\in L^2(\nu)$ we get the well definition of the integrals.

We next observe that the sequence $(\rho_m(x))_{m\geq1}$ is nondecreasing for any $x\in\R$. Therefore, the sequence  $\big(\frac{1}{2} \lambda \big[p \sigma  - (z + \frac{C_{\kappa, \eta}}{\lambda} ) \big]^2\rho_m(z)\big)_{m\geq1}$ is nondecreasing for any $z\in\R$ and any $p\in [-\underline \pi,\overline\pi]$.

The sequence $(h_\lambda(\varphi_m(x)))_{m\geq1}$ is also nondecreasing for any $x\in \R$. We indeed have 
\beqs
h_\lambda(\varphi_{m+1}(x))-h_\lambda(\varphi_{m}(x)) & = &\mathds{1}_{x\geq m}\big(h_\lambda(\varphi_{m+1}(x))-h_\lambda(\varphi_{m}(x))\big)\;.
\enqs
Since $\varphi_{m+1}(x)= \varphi_{m}(x)$ for $x<m$, $\varphi_{m+1}(x)\geq \varphi_{m}(x)\geq 0$ for $x\geq m$, and $h_\lambda$ is nondecreasing on $\R_+$, we get
\beqs
h_\lambda(\varphi_{m+1}(x))-h_\lambda(\varphi_{m}(x)) & \geq & 0
\enqs 
for all $x\in\R$. Therefore we get $f_{m+1}\geq f_m$ by passing to the infimum over $p\in[\underline \pi,\overline \pi]$.

(ii) For $m\geq1$, we have from the truncation term $\rho_m$ that the function $f^{1,m}(.,.,p)$ is Lipschitz continuous: there exists a continuous and positive function $p\mapsto L_m(p)$ such that
    \beqs
        |f^{1}_m(z,u,p)-f^1_{m}(z',u',p)| & \leq  & L_{m}(p)\Big(|z-z'|+|u-u'|_{L^2(\nu)}\Big)
    \enqs
    for all $z,z'\in\R$ and $u,u'\in L^2(\nu)$.
 Therefore, the function $(u,z)\mapsto \inf_{\pi\in[ -\underline \pi, \overline \pi]}f^1_m(z,u,p)$ is Lipschitz continuous as an infimum of Lipschitz functions over a compact set. 

By the same argument, we get that $(u,z)\mapsto \int_{\gamma(\R)\setminus\{0\}}\inf_{p\in [-\underline \pi,\overline \pi]}f^2_m(z,u,g,p)\mu(dg)$ is also Lipschitz continuous.  
Hence $f_{m}$ is Lipschitz continuous for $m\geq1$.
% Fix $u \in L^2(\nu)$. For $ m \geq 1$, the function $z \mapsto \frac{1}{2} \lambda \left[p - \left(z + \frac{C_{\kappa, \eta, K}}{\lambda} \right) \right]^2\rho_m(z)$ is $\mathcal{C}^1$ with bounded  first derivative due to the truncation induced by the function $\rho$. So this function is Lipschitz continuous. 
% Fix now $z\in\R$. Using the same argument, we have also Lipschitz property in $u$. 
\end{proof}

We next introduce for $m\geq1$ the  processes $(Y^{m}, Z^{m}, U^{m})_{m}\in S^\infty \times L^2(W) \times L^2(\tilde{N})$ solution to the BSDE
\begin{equation}\label{BSDEexpm}
Y_t = F + \int_{t}^{T} f_{m}( Z_s, U_s) \, ds - \int_{t}^{T} Z_s \, dW_s - \int_{t}^{T} \int_{\R} U_s(e) \, \tilde{N}(ds, de), \quad t \in [0, T].   
\end{equation}
Lemma \ref{LemLipfm} ensures the existence and uniqueness of the process $(Y^{m}, Z^{m}, U^{m})_{m\geq1}\in S^\infty \times L^2(W) \times L^2(\tilde{N})$ for ${m}\geq1$.

We next provide the following estimate on the functions $f^m$, $m\geq 1$.

\begin{Lemma}\label{LemQuadGrowfm} 
For $m\geq1$, the functions %$f$ and 
 $f_{m}$ satisfies the following bounds%:
%\beqs
% \frac{\lambda}{2} z^2 + |u|_\lambda\geq f(z, u) & \geq & -z\int_{\gamma= 0}\eta(e)\nu(de)-\left(\frac{C_{\kappa,\eta,K}}{\lambda}\int_{\gamma= 0}\eta(e)\nu(de)\right)\int_{\gamma= 0}\eta(e)\nu(de)\\
%  & & -(\underline \pi+\overline \pi) %\int_{\gamma(\R)\setminus\{0\}}|\hat \eta(g)|\mu(dg)
%\enqs
%and
\beqs
 \frac{\lambda}{2} z^2 + |u|_\lambda~~\geq~~ f_{m}(z, u) & \geq &  -zC_{\kappa, \eta} - \frac{C_{\kappa, \eta}^2}{2\lambda}%\\
  %& & 
  -(\underline \pi+\overline \pi) \int_{\R}| \eta(e)|\nu(de)
\enqs
for all  $z\in \mathbb{R}$ and $u \in L^{\infty}(\nu)\cap L^{2}(\nu)$. 
\end{Lemma}

\begin{proof}
Let \( z \in \mathbb{R} \), \( u \in L^{\infty}(\nu) \). 
Since $ 0 \in [- \underline{\pi}, \overline \pi]$ and $h_\lambda\geq 0$, we get from the definition of $f_m$
\beqs
f_m(z, u) & \leq & \frac{\lambda}{2} \left( z + \frac{C_{\kappa, \eta}}{\lambda} \right)^2 + \int_\R \left( \frac{e^{\lambda u(e)} - 1}{\lambda} - u(e) \right) \nu(de) - z C_{\kappa, \eta} - \frac{C_{\kappa, \eta}^2}{2\lambda}.
\enqs
By expanding the square in the first term of the r.h.s., we get the upper bound.

For the lower bound, we have
\beq\label{lbdfg}
\frac{\exp(\lambda(u(e) - p \eta(e))) - 1}{\lambda} - u(e) & = & h_{\lambda}\big(u(e) - p \eta(e)\big) -  p\eta(e).
\enq
for all $ e \in \R$. Since  $h_\lambda$ is a positive function we get
\beqs
f(z,u) & \geq & \inf_{ p \in[ -\underline{\pi},\overline{\pi}]} \tilde f^{1}(z,u,p)
%\label{defDriverfm}
 + \int_{\gamma(\R) \setminus \{0\}} \inf_{p \in [-\underline{\pi}, \overline{\pi}]} \tilde f^{2}(z,u,g,p)
\mu(dg) %\\
%& & 
 - z C_{\kappa, \eta} - \frac{C_{\kappa, \eta}^2}{2 \lambda}
\enqs
with
\beqs
\tilde f^1(z, u,p) & = & %\frac{1}{2} \lambda \left[p - \left(z + \frac{C_{\kappa, \eta, K}}{\lambda} \right) \right]^2 %\nonumber \\
 %& & 
 -|p| \int_{\gamma = 0}  |\eta(e)|\nu(de) 
\enqs and
\beqs
\tilde f^{2}(z, u,g,p) & = & - p\hat{\eta}(g)\;. %\times K(g,\R)
%\int_{\gamma = g} \left[ \frac{\exp(\lambda(u(e) - p \eta(e))) - 1}{\lambda} - 
%u(e)  \right]\rho_m(u(e)) K(g, de) 
\enqs 
In particular, we get
\beq\label{lbdft1}
\inf_{p \in [-\underline{\pi}, \overline{\pi}]} \tilde f^1(z, u, p)
 & \geq & - z C_{\kappa, \eta} - \frac{C_{\kappa, \eta}^2}{2 \lambda}-(\underline \pi+\overline \pi)\int_{\gamma=0}|\eta(e)|\nu(de)\;. 
 %-z\int_{\gamma= 0}\eta(e)\nu(de)-\left(\frac{C_{\kappa,\eta,K}}{\lambda}\int_{\gamma= 0}\eta(e)\nu(de)\right)\int_{\gamma= 0}\eta(e)\nu(de)\;.
\enq
Still using the positivity of $h_\lambda$, % and since $\hat \eta(g)<0$ for $\mu$-a.a. $g\in\gamma(\R)\setminus\{0\}$, 
we get 
\beq\nonumber\label{lbdft2}
\int_{\gamma(\R)\setminus\{0\}}\inf_{p\in[ -\underline \pi, \overline \pi]} \tilde f^2(z, u,g,p)\mu(dg) & \geq &%\\
\int_{\gamma(\R)\setminus\{0\}}\inf_{p\in[ -\underline \pi, \overline \pi]}-p\int_{\gamma=g}\eta(e)K(g,de) \\ %& \geq &
%-(\underline \pi+\overline \pi) \int_{\gamma(\R)\setminus\{0\}}|\hat \eta(g)|\mu(dg)\\
 & \geq &-(\underline \pi+\overline \pi) \int_{\gamma(\R)\setminus\{0\}}|\eta(e)|\nu(de) \;.
\enq
We then get the lower bound for $f_m$ from \reff{lbdfg}, \reff{lbdft1} and \reff{lbdft2}.
%We conclude with positivity of the functional $g_\lambda$ and the lower bounds $p$.\\ 
%Same property  holds for $f^m$ using bounds of $\rho_m$.
%We turn to $f_{m}$. We first notice that $f_{m}\leq f$, therefore we get the same upper bound for $f_{m}$ as $f$.
%For the lower bound, the same arguments as for $f$ can be applied for $f_m$ as the penalization concerns only positive terms involving the function $h_\lambda$.
\end{proof}

We next present a result on $f$ allowing to compare the processes $Y^m$ for $m\geq1$.

\begin{Lemma}\label{LemComp} %Suppose that
%\beq\label{condexp-mom-eta}
%\int_{\R}\mathds{1}_{\eta(e)>1}e^{C\eta(e)}\nu(de) %& < & +\infty
%\enq
%for all $C>0$. 
For $m\geq1$, there exists a function $k_{m}:~L^\infty(\nu)\times L^\infty(\nu)\times \R\rightarrow \mathbb{R}$, %non-negative constant $C_3, \alpha$ 
satisfying the following properties.
\begin{enumerate}[(i)]

\item  The function $f_{m}$ satisfies 
\beqs
     f_{m}(z,u) - f_{m}(z,u') & \leq &  \int_{\R}k_{m}(u, u')(e)(u(e) - u'(e))\nu(de)
 \enqs
 for all $ z \in \mathbb{R}$  and $u, u' \in  L^{\infty}(\nu)\cap L^{2}(\nu)$. 
%\begin{equation}
 %   |f(z, u) - f(z', u)|\leq C_3( \alpha + |z| + |z'|)|z-z'|
%\end{equation}
%for all $z, z' \in \mathbb{R}, u \in L^{\infty}(\nu)$.
\item For each $M>0$, there exist two constants $\bar C_{M,m},\underline C_{M,m}>0$ such that
\beqs
\big(-1 + \underline{C}_{M,m}%e^{-\lambda (n+\underline{\pi})|\eta(e)|}
\big)\mathds{1}_{|e|>\frac{1}{m}}  &\leq & k_m(u, u')(e) ~~\leq~~\mathds{1}_{|e|>\frac{1}{m}}\bar C_{M,m}%\;,\quad e\in \R  \;,
\enqs
%and 
%\beqs
%\|k(u,u')\|_{L^2(\nu)} & \leq & \bar C_M
%\enqs
for all $u, u' \in  L^{\infty}(\nu)$  such that  $\|u\|_{L^{\infty}(\nu)},\, \|u'\|_{L^{\infty}(\nu)} \leq M$ and all $e\in \R  $.
\end{enumerate}
\end{Lemma}
\begin{proof} 

Let $ z \in \mathbb{R} $ and  $u, u' \in L^{\infty}(\nu(de))$. Using the inequality
\beqs
\inf_{\pi} A^\pi - \inf_{\pi} B^\pi & \leq & \sup_{\pi} (A^\pi - B^\pi)
\enqs
we get
\beqs
f_{m}(z, u) - f_{m}(z, u')  & \leq &   \sup_{ p\in [-\underline{\pi},\overline{\pi}] } \int_{\gamma = 0} \Delta^{u,u'}(e,p) \nu(de) 
\\
 & &   + \int_{\gamma(\R) \setminus \{0\}} \sup_{p\in [-\underline{\pi},\overline{\pi}]}\int_{\gamma =g}\Delta^{u,u'}(e,p) K(g, de) %\mathds{1}_{|e|>\frac{1}{m}}
 \mu(dg).
\enqs
with 
\beqs
\Delta^{u,u'}(e,p) & = & \mathds{1}_{|e|>\frac{1}{m}}\left(h_{\lambda}\Big(\varphi_m\big(u(e)-p\eta(e)\big)\Big) - h_{\lambda}\Big(\varphi_m\big(u'(e)-p\eta(e)\big)\Big) \right)
\enqs
for $(e,p)\in \R\times  \R$. 
From Taylor's formula, we have
\beqs
\Delta^{u,u'}(e,p) & = &  
\delta^{u,u'}(p,e)
\left(u(e)-u'(e)\right)
\enqs
with
\beqs
\delta^{u,u'}(p,e) & = & \mathds{1}_{|e|>\frac{1}{m}}\int_0^1 \varphi_m'\big((\alpha u+ (1-\alpha) u')(e)\big)h'_\lambda \left(\varphi_m\big( (\alpha u+ (1-\alpha) u' - p\eta)(e) \big)\right)  d\alpha\;.
\enqs
Since $\varphi_m$ and $\varphi_m'$ are bounded, and $\|u\|_{L^{\infty}(\nu)},\, \|u'\|_{L^{\infty}(\nu)} \leq M$, there exists a constant $C_M$ depending on $M$ such that
\beqs
|\delta^{u,u'}(p,e)| & \leq & \mathds{1}_{|e|>\frac{1}{m}}C_M(|u(e)|+|u'(e)|+|p||\eta(e)|)\;,\quad e\in \R  \;.
\enqs
Finally, we get
\beqs
f_{m}(z, u) - f_{m}(z, u')  & \leq &  \sup_{ p\in[ -\underline{\pi} , \overline{\pi} ]} \int_{\gamma = 0} \delta^{u,u'}(p,e)
\left(u(e)-u'(e)\right) \nu(de) 
\\
 & &   + \int_{\gamma(\R) \setminus \{0\}} \sup_{\overline{\pi}\geq p\geq -\underline{\pi}}\int_{\gamma =g} \delta^{u,u'}(\alpha,p,e)
\left(u(e)-u'(e)\right) K(g, de) \mu(dg)\\
 & \leq & \int_\R k_m(u,u')(e)(u(e)-u'(e))\nu(de) 
\enqs
with
\beqs
k_{m}(u,u')(e) & = & \mathds{1}_{|e|>\frac{1}{m}}\left[\mathds{1}_{u(e)\geq u'(e)}\left(\sup_{ p\in[ -\underline{\pi}, \overline{\pi} ]}\delta^{u,u'}(p,e) \right)+\mathds{1}_{u(e)\leq u'(e)}\left(\inf_{ p\in[ -\underline{\pi}, \overline{\pi} ]}\delta^{u,u'}(p,e) \right)\right]
\enqs
for $e\in \R  $. 
From the definition of $\varphi_m$,  and for $M>0$ such that $|u|_{L^\infty(\nu)},|u'|_{L^\infty(\nu)}\leq M$, there exists a constant $C_M>0$ such that
\beqs
1~~\geq~~\varphi_m'\big((\alpha u+ (1-\alpha) u')(e)\big) & \geq & C_M
\enqs
for $\nu$-a.a. $e\in \R$. Therefore, we get 
\beqs
\mathds{1}_{|e|>\frac{1}{m}}\left(\int_0^1 \exp \left(\lambda\varphi_m\big( (\alpha u+ (1-\alpha) u' - p\eta)(e) \big)\right)  d\alpha\right) & \geq & \delta^{u,u'}(p,e) 
\enqs
and
\beqs
\delta^{u,u'}(p,e) & \geq & \mathds{1}_{|e|>\frac{1}{m}}\left(C_M\int_0^1 \exp \left(\lambda\varphi_m\big( (\alpha u+ (1-\alpha) u' - p\eta)(e) \big)\right)  d\alpha-1\right)\;.
\enqs
%Still using $\|u\|_{L^{\infty}(\nu)},\, \|u'\|_{L^{\infty}(\nu)} \leq M$ we get from the definition of $\varphi_m$ 
Since $\varphi_m$ is upper bounded, there exists two constant $\overline C_{M,m},\underline C_{M,m}>0$ such that
\beqs
\overline C_{M,m}\mathds{1}_{|e|>\frac{1}{m}}~~\geq~~\delta^{u,u'}(p,e) & \geq & \mathds{1}_{|e|>\frac{1}{m}}\left(\underline{C}_{M,m}  -1\right)
\enqs
for all $p\in[-\underline \pi,\overline\pi]$ and all $e\in \R$. From the definition of $k_m$ we get
\beqs
\overline C_{M,m}\mathds{1}_{|e|>\frac{1}{m}}~~\geq~~k_m(u, u')(e) &\geq &  \big(-1 + \underline{C}_{M,m}\big)\mathds{1}_{|e|>\frac{1}{m}}   \;,\quad e\in \R.
\enqs
\end{proof}
% Using \eqref{BorneUnifYm}, we deduce that it converges and we denote by $\tilde Y$ its limit.   

We now have the following result, which firstly provides the monotonicity of the sequence $(Y^m)_{m\geq 1}$, and secondly a uniform bound on the sequence $(Y^m,Z^m,U^m)_{m\geq 1}$ similar to 
% Lemma \ref{LemQuadGrowfm} and 
  Lemma 3 in \cite{MAM08}
 \begin{Proposition}\label{Prop ymbdd} (i) The sequence $(Y^m)_{m\geq 1}$ is nondecreasing in the following sense
 \beqs
 Y^{m+1}_t & \geq & Y^m_t\;,\quad \P-a.s.
 \enqs
 for all $t\in[0,T]$ and all $m\geq1$.

\noindent (ii) There exists a constant $C$ such that
\beq\label{BorneUnifYm}
    |Y_t^{m}|  & \leq & C\quad ~\P-a.s.
\enq
for all $t\in[0,T]$, and 
\beq\label{BorneUnifZUm}
\E\Big[\int_0^T|Z_s^{m}|^2ds+\int_0^T\int_\R|U^{m}_s(e)|^2\nu(de)ds\Big]   & \leq & C
\enq
for all $m\geq1$.
\end{Proposition}

\begin{proof}
(i) We first observe that the sequence $(f_{m}(z,u))_{m\geq1}$ is nondecreasing for each $(z,u)\in \R\times L^2(\nu)$. From Lemmata \ref{LemLipfm} and \ref{LemComp}, we can apply comparison Theorem \ref{ComparisonBSDEJ}, and we get that the sequence $(Y^{m})_{m\geq1}$ is a.s. nondecreasing.

(ii) Define $(\underline Y,\underline Z,\underline U)\in S^\infty \times L^2(W) \times L^2(\tilde{N}) $ as the solution to the BSDE
\beqs
\underline Y_t = F + \int_{t}^{T} \underline f( \underline Z_s,\underline U_s) \, ds - \int_{t}^{T}\underline Z_s \, dW_s - \int_{t}^{T} \int_{\R} \underline U_s(e) \, \tilde{N}(ds, de), \quad t \in [0, T]\;,   
\enqs
where $\underline f$ is defined by
\beqs
\underline f (z,u) & = &  -zC_{\kappa, \eta} - \frac{C_{\kappa, \eta}^2}{2\lambda}%\\
  %& & 
  -(\underline \pi+\overline \pi) \int_{\R}| \eta(e)|\nu(de)
%
%-z\int_{\gamma= 0}\eta(e)\nu(de)-\left(\frac{C_{\kappa,\eta,K}}{\lambda}\int_{\gamma= 0}\eta(e)\nu(de)\right)\int_{\gamma= 0}\eta(e)\nu(de)\\
%  & & -(\underline \pi+\overline \pi) \int_{\gamma(\R)\setminus\{0\}}|\hat \eta(g)|\mu(dg)
\enqs
for all $(z,u)\in\R\times L^2(\nu)$. The triple $(\underline Y,\underline Z,\underline U)\in S^\infty \times L^2(W) \times L^2(\tilde{N})$ is well defined since $F$ is bounded and $\underline f$ is Lipschitz continuous. %Moreover, we have the following explicit expression for the component $Y$. Since  $\underline f$ does not involve the component $u$, 
Since $\underline f$ does not depend on the variable $u$, we can apply comparison Theorem \ref{ComparisonBSDEJ}, which gives
\beq\label{LbdYm}
Y^m_t & \geq & \underline Y_t \qquad\P-a.s.
\enq
for all $t\in[0,T]$ and all $m\geq 1$.

To get an upper bound, we first apply the martingale representation theorem to the process  (see e.g. \cite{Kunita04}).
\beqs
\check Y_t & := & \E\big[ {}\exp(\lambda F)+1|\Fc_t\big]\;,\quad t\in[0,T]\;.
\enqs
 Therefore, there exists
 $(\check Z,\check U)\in L^2(W) \times L^2(\tilde{N})$such that
 \beqs
 \check Y_t  & = & \exp(\lambda F)+1-\int_t^T \check Z_s dW_s-\int_t^T\int_\R \check U_s(e) \tilde N(de,ds)\;,\quad t\in[0,T]\;.
 \enqs
 Applying, It\^o's formula to the process $\frac{1}{\lambda} \log(\check Y)$ we get that the triple $(\overline Y,\overline Z,\overline U)\in S^\infty \times L^2(W) \times L^2(\tilde{N}) $ defined by
 \beqs
 \overline Y_t & = & \frac{1}{\lambda} \log(\check Y_t)\\
 \overline Z_t & = & \frac{Z_t}{\lambda \check Y_t}\\
 \overline U_t(e) & = & \log\Big( \frac{\check U_t(e)}{Y_{t-}} +1\Big)
 \enqs
%
%We now define Define $(\overline Y,\overline Z,\overline U)\in S^\infty \times L^2(W) \times L^2(\tilde{N}) $ as the solution to the BSDE
for $(t,e)\in[0,T]\times \R$, satisfies
\beqs
\overline Y_t = \overline F + \int_{t}^{T} \overline f( \overline Z_s,\overline U_s) \, ds - \int_{t}^{T}\overline Z_s \, dW_s - \int_{t}^{T} \int_{\R} \overline U_s(e) \, \tilde{N}(ds, de), \quad t \in [0, T]\;,   
\enqs
where 
\beqs
\overline F & := & \frac{1}{\lambda}\log\Big(\exp\big(\lambda F\big)+1\Big)
\enqs
and $\overline f$ is defined by
\beqs
\overline{f}(z,u) & = & \frac{\lambda}{2} z^2 + |u|_\lambda
\enqs
for all $(z,u)\in\R\times L^2(\nu)$. We observe that $\overline F\ge F$ and from Lemma \ref{LemQuadGrowfm} we have $\overline f\geq f_m$. From Lemmata \ref{LemLipfm} and \ref{LemComp}, we can apply comparison Theorem \ref{ComparisonBSDEJ} and we get 
\beq\label{UbdYm}
Y^m_t & \leq & \overline Y_t \qquad\P-a.s.
\enq
for all $t\in[0,T]$ and all $m\geq 1$.
From \reff{LbdYm} and \reff{UbdYm}, we get \reff{BorneUnifYm}.

We turn to \reff{BorneUnifZUm}. For that, we apply It\^o's formula to $(Y^m-C)^2$ where $C$ is a constant satisfying \reff{BorneUnifYm}. We then get
\beq\nonumber
\E\big[(Y^m_0-C)^2\big] & = & \E\big[(F-C)^2\big]+\E\Big[\int_0^T2(Y^m_s-C)f(Z^m_s,U^m_s)ds\Big]\\
 & &  -\E\Big[\int_0^T|Z^m_s|^2ds\Big]-\E\Big[\int_0^T \int_\R |U^m_s(e)|^2\nu(de)ds\Big]\;.\label{ItoestimZmUm}
\enq
From Lemma \ref{LemQuadGrowfm}, there exists two constants $C_1$ and $C_2$ such that
\beqs
f_m(z,u) & \geq & C_1z+C_2
\enqs
for all $z\in\R$, $u\in L^2(\nu)$ and $m\geq1$. Since $C$ satisfies \reff{BorneUnifYm}, we have
\beqs
0~~\geq~~(Y^m_s-C) & \geq & -2C\;,\quad s\in[0,T]\;.
\enqs
Therefore, we get
\beqs
(Y^m_s-C)f(Z^m_s,U^m_s) & \leq & -2C(C_1Z^m_s+C_2)\;,\quad s\in[0,T]\;.
\enqs
From Young's inequality, we get
\beqs
(Y^m_s-C)f(Z^m_s,U^m_s) & \leq &\frac{1}{2}|Z^m_s|^2+2C|C_2|+2|CC_1|^2\;.
\enqs
Plugging this inequality into \reff{ItoestimZmUm}, we get
\beqs
 \frac{1}{2}\E\Big[\int_0^T|Z^m_s|^2ds\Big]+\E\Big[\int_0^T \int_\R |U^m_s(e)|^2\nu(de)ds\Big] & \leq & 4C^2+2T(C|C_2|+|CC_1|^2)
\enqs
for all $m\geq1$.
\end{proof}

\begin{Lemma}\label{lem-ref-f} The function $f$ satisfies the following properties.
\begin{enumerate}[(i)]
\item  There exists a constant $C$ such that 
\beq\label{locLipfz}
    |f(z, u) - f(z', u)| & \leq &  C( 1 + |z| + |z'|)|z-z'|
\enq
for all $z, z' \in \mathbb{R}, u \in L^{2}(\nu)\cap L^{\infty}(\nu)$.
\item The function $u\mapsto f(z,u)$ is continuous on $L^\infty(\nu)\cap L^2(\nu)$ in the following sense: for any sequence $(u^m)\in L^2(\nu)\cap L^\infty(\nu)$ converging in  $L^2(\nu)$ to $u\in L^2(\nu)\cap L^\infty(\nu)$ and bounded in $L^\infty(\nu)$ we have %$
%for the norm $|\cdot|_{L^2(\nu)}$ for any $z\in\R$.
\beqs
f(z,u^m) & \xrightarrow[m\rightarrow+\infty]{} & f(z,u)
\enqs
for any $z\in\R$.
 \end{enumerate}
\end{Lemma}
\begin{proof} %\textcolor{red}{We need compactness to compute sup here}

    (i) Let $u \in L^{2}(\nu)\cap L^{\infty}(\nu)$ and $z, z' \in \mathbb{R}$.\\
We have 
\begin{align*}
f(z,u) - f(z',u) 
& \leq \frac{\lambda}{2}\sup_{p\in[-\underline \pi,\overline \pi]} \Big( [p \sigma  - (z + \tfrac{C_{\kappa, \eta}}{\lambda})]^2 - [p \sigma - (z' + \tfrac{C_{\kappa, \eta}}{\lambda})]^2 \Big) - C_{\kappa, \eta}(z-z') \\
& \leq \frac{\lambda}{2} |z-z'| \sup_{p\in[-\underline \pi,\overline \pi]} \big|2p\sigma - (z+z' + \tfrac{2C_{\kappa, \eta}}{\lambda})\big| + |C_{\kappa, \eta}|\,|z-z'|.
\end{align*}
Since $p \in [-\underline \pi,\overline \pi]$, and applying the same argument to $f(z',u) - f(z,u)$ yields \eqref{locLipfz} with 
\beqs
C & = & \frac{\lambda}{2}[2\sigma (\underline \pi+\overline \pi)+1]+ 2 C_{\kappa, \eta}\;.
\enqs
    
    %Let $u \in L^{2}(\nu)\cap L^{\infty}(\nu)$ and $z, z' \in \mathbb{R}$.\\We have \beq f(z,u) - f(z',u) & \leq &  \frac{\lambda}{2}\sup_{p\in[-\underline \pi,\overline \pi]} ([p \sigma  - (z + \frac{C_{\kappa, \eta, K}}{\lambda})]^2 - [p \sigma - (z' + \frac{C_{\kappa, \eta, K}}{\lambda})]^2) - C_{\kappa, \eta, K}(z-z')\\ & \leq & \frac{\lambda}{2}\sup_{p\in[-\underline \pi,\overline \pi]} [2p - (z+z')][z-z'] + C_{\kappa, \eta, K}(1-\lambda)(z-z')\\& \leq & \Big(\frac{\lambda}{2}[2(\underline \pi+\overline \pi)+|z|+|z'|)+C_{\kappa, \eta, K}(1+\lambda)\Big)|z-z'|.\enqs Applying the same argument for $f(z',u) - f(z,u)$ we get \reff{locLipfz}  with \beqs C & = & \frac{\lambda}{2}[2(\underline \pi+\overline \pi)+1]+C_{\kappa, \eta, K}(1+\lambda)\;.\enqs

    (ii) Fix $z\in\R$ and $(u^m)_m$ a sequence of $L^{2}(\nu)\cap L^{\infty}(\nu)$ bounded in $L^{\infty}(\nu)$ and converging in $L^{2}(\nu)$ to $u \in L^{2}(\nu)\cap L^{\infty}(\nu)$. Since $\eta\in(-1,0]$, we notice that in the definition of $f$, we can restrict the first infimum to $p\in[-\underline{\pi},(z+\frac{C_{\kappa,\eta}}{\lambda})\vee (-\underline{\pi})]$.
    We also notice that the second infimum is obtained for $p=-\underline\pi$. 
    From the inequality
\begin{align}
\inf_{\pi} A^\pi - \inf_{\pi} B^\pi \leq \sup_{\pi} (A^\pi - B^\pi)
\end{align}
we get
\beqs
    f(z,u)-f(z,u^m) & \leq & \sup_{p\in [-\underline \pi,\overline \pi]}\int_{\gamma=0} \Delta^m(e,p) \nu(de)
    \\
     & & + \int_{\gamma(\R)\setminus \{0\}}\sup_{p\in [-\underline \pi,\overline \pi]}\int_{\gamma=g} \Delta^m(e,p) K(g,de)\mu(dg)
    \\
\enqs
where 
\beqs
\Delta^m(e,p) & = & |h_\lambda(u(e)+p\eta(e))-h_\lambda(u^m(e)+p\eta(e))|\;,\quad e,p\in\R\;.
\enqs
Since $(u^m)_{m\geq1}$ is bounded in $L^{\infty}(\nu)$ and $u\in L^\infty(\nu)$, there exists a constant $C$ such that
\beq\label{ubdd}
|u(e)-p\eta(e)| & \leq & C\\
|u^m(e)-p\eta(e)| & \leq & C \label{umbdd}
\enq
for all $e\in\R$, $p\in [-\underline \pi,\overline \pi]$ and $m\geq1$.

Fix $\eps>0$. Since $h_\lambda$ is continuous, it is uniformly continuous on $[-C,C]$. Therefore, there exists some $\alpha\in(0,1)$ such that
\beqs
\sup_{|x-y|\leq \alpha}|h_\lambda(x)-h_\lambda(y)| & \leq & \eps\;.
\enqs
Moreover, by applying Taylor's formula to $h_\lambda$, there exists a constant $C'$ such that
\beqs
|h_\lambda(x)| & \leq & C'|x|^2\;,\quad x\in {[-C,C]}\;.
\enqs
We therefore have
\beqs
\sup_{p\in [-\underline \pi,\overline \pi]}\int_{\gamma=0} \Delta^m(e,p) \nu(de) & \leq & A^m+B^m 
\enqs
where
\beqs
A^m & = & \sup_{p\in [-\underline \pi,\overline \pi]}\int_{\gamma=0}\mathds{1}_{|u-u^m|\leq\alpha}(e)\eps\wedge\big(|u(e)-p\eta(e)|^2+|u^m(e)-p\eta(e)|^2\big)\nu(de)\;,\\
B^m & = & \sup_{p\in [-\underline \pi,\overline \pi]}\int_{\gamma=0}\mathds{1}_{|u-u^m|>\alpha}(e)|h_\lambda (u(e)-p\eta(e))-h_\lambda(u^m(e)-p\eta(e))|\nu(de)\;.
\enqs
From \reff{ubdd}-\reff{umbdd} we have
\beq\label{convBm}
B^m & \leq & 2\sup_{[-C,C]}|h_\lambda|\int_{\gamma=0}\frac{|u(e)-u^m(e)|^2}{\alpha^2}\nu(de)~~ \xrightarrow[m\rightarrow+\infty]{}~~0\;.
\enq
We next have
\beq\label{estimAm}
A_m & \leq & 4\int_{\gamma=0} \eps\wedge |u(e)|^2\nu(de)+4(\underline \pi^2+\overline \pi^2)\int_{\gamma=0}\eps\wedge |\eta(e)|^2\nu(de)+\int_\R|u(e)-u^m(e)|^2\nu(de)\qquad\qquad 
\enq
Therefore, we get from \reff{convBm}, \reff{estimAm} and the convergence of $(u^m)_m$ to $u$ in $L^2(\nu)$
\beqs
\limsup_{m\rightarrow+\infty}\sup_{p\in [-\underline \pi,\overline \pi]}\int_{\gamma=0} \Delta^m(e,p) \nu(de)  & \leq & 4\int_{\gamma=0} \eps\wedge |u(e)|^2\nu(de)+4(\underline \pi^2+\overline \pi^2)\int_{\gamma=0}\eps\wedge |\eta(e)|^2\nu(de)\;.
\enqs
for any $\eps>0$. By the dominated convergence theorem, we have
\beqs
4\int_{\gamma=0} \eps\wedge |u(e)|^2\nu(de)+4(\underline \pi^2+\overline \pi^2)\int_{\gamma=0}\eps\wedge |\eta(e)|^2\nu(de) & \xrightarrow[\eps\rightarrow0+]{ } & 0\;.
\enqs
Therefore, we get 
\beqs
\limsup_{m\rightarrow+\infty}\sup_{p\in [-\underline \pi,\overline \pi]}\int_{\gamma=0} \Delta^m(e,p) \nu(de)  & \leq & 0\;.
\enqs
By the same argument we get
\beqs
\limsup_{m\rightarrow+\infty}\int_{\gamma(\R)\setminus \{0\}}\sup_{p\in [-\underline \pi,\overline \pi]}\int_{\gamma=g} \Delta^m(e,p) K(g,de)\mu(dg)  & \leq & 0\;,
\enqs
which gives
\beqs
\limsup_{m\rightarrow+\infty}   f(z,u)-f(z,u^m)  & \leq & 0\;.
\enqs
We finally notice that the same arguments can be applied to $f(z,u^m)-f(z,u)$  which gives 
\beqs
\lim_{m\rightarrow+\infty}   f(z,u)-f(z,u^m)  & = & 0\;.
\enqs
\end{proof}

We are now able to use the following stability result similar to Lemma 5 in \cite{MAM08}.
\begin{Proposition}\label{stabilityBSDE}
    Let $(\mathfrak{f}^m)_{m\geq1}$ be a sequence of functions from $\R\times L^2(\nu)$ to $\R$ satisfying the following conditions.
    \begin{enumerate}[(i)]
        \item $(\mathfrak{f}^m)_{m\geq1}$ is nondecresing and converges to some measurable function $\mathfrak{f}$ in the following sense:
        \beqs
\mathfrak{f}^m(z^m,u^m)\xrightarrow[m\rightarrow+\infty]{}\mathfrak{f}(z,u)
        \enqs
        for any sequence $(z^m,u^m)_{m\geq1}$ of $\R\times L^2(\nu)\cap L^\infty(\nu)$ converging in  $\R\times L^2(\nu)$ to $(z,u)$ in $\R\times L^2(\nu)\cap L^\infty(\nu)$ with $(u^m)_m$  bounded in $L^\infty(\nu)$.
        \item There exists constants $C_1,C_2>0$ such that
       % \beq\label{lbdg}
%- z C_1 - C_2%-C_3 |u|_{L^2(\nu)}
%& \leq & \mathfrak{f}(z, u) ~~\leq~~ \frac{\lambda}{2} z^2 + |u|_\lambda,\quad (z,u)\in\R\times L^2(\nu)\;,
%\enq
%and 
\beq\label{lbdgm}
- z C_1 - C_2%-C_3 |u|_{L^2(\nu)}
& \leq & \mathfrak{f}_m(z, u) ~~\leq~~ \frac{\lambda}{2} z^2 + |u|_\lambda %,%\quad (z,u)\in\R\times L^2(\nu)\;,
\enq
for all $m\geq1$ and all $(z,u)\in \R\times L^2(\nu)\cap L^\infty(\nu)$\;.
\item For $m\geq 1$, there exists $(\tilde Y^m,\tilde  Z^m,\tilde  U^m)\in S^\infty \times L^2(W) \times L^2(\tilde{N})$ solution to the BSDE
\beqs%\label{BSDEexpm}
\tilde Y_t  & =&  F + \int_{t}^{T} \mathfrak{f}^m( \tilde Z_s,\tilde  U_s) \, ds - \int_{t}^{T}\tilde  Z_s \, dW_s - \int_{t}^{T} \int_{\R}\tilde  U_s(e) \, \tilde{N}(ds, de), \quad t \in [0, T].    
\enqs
\item The sequence $(\tilde Y^m)_{m\geq1}$ is nondecreasing and there exists a constant $C$ such that
\beqs
    |\tilde Y_t^{m}|  & \leq & C\quad ~\P-a.s. %~t\in[0,T]\;,\\
\enqs
for all $t\in[0,T]$, and 
\beqs
\E\Big[\int_0^T|\tilde  Z_s^{m}|^2ds+\int_0^T\int_\R|\tilde  U^{m}_s(e)|^2\nu(de)ds\Big]   & \leq & C
\enqs
for all $m\geq1$.
    \end{enumerate}
Then, $(\tilde Y^m, \tilde Z^m, \tilde U^m)_m$ converges  in $S^2 \times L^2(W) \times L^2(\tilde{N})$ to $(\tilde Y, \tilde Z, \tilde U)\in S^\infty \times L^2(W) \times L^2(\tilde{N})$  solution to the BSDE
\beqs
   \tilde  Y_t  & = &  F + \int_{t}^{T} \mathfrak{f}( \tilde Z_s, \tilde U_s) \, ds - \int_{t}^{T}\tilde  Z_s \, dW_s - \int_{t}^{T} \int_{\R} \tilde 
   U_s(e) \, \tilde{N}(ds, de), \quad t \in [0, T].
\enqs
\end{Proposition}
\begin{proof}
 % We define $\alpha=\nu(\gamma\neq0\cap\{\eta\leq1\})$ and $(\tilde Y^m_t,\tilde Z^m_t,\tilde U^m_t)=(e^{\alpha t}\tilde Y^m_t,e^{\alpha t}\tilde Z^m_t,e^{\alpha t}\tilde U^m_t)$ for $t\in[0,T]$ and $m\geq1$. Then $(\tilde Y^m_t,\tilde Z^m_t,\tilde U^m_t)$ satisfies
 % \begin{align}
 %     \tilde Y^m_t = \tilde F + \int_{t}^{T} \tilde g^m(s, Z_s, U_s) \, ds - \int_{t}^{T} \tilde Z ^m_s \, dW_s - \int_{t}^{T} \int_{\R} \tilde U ^m_s(e) \, \tilde{N}(ds, de), \quad t \in [0, T].    
 % \end{align}
 % with $\tilde F=e^{\alpha T}F$ and $\tilde g^m(t,z,u)=g^m(z,u)+{\alpha}$
The only difference with Lemma 5 in \cite{MAM08} is the lower bound in %\reff{lbdg} and 
\reff{lbdgm} where it is a specific affine function of the variable $z$. However, the arguments used in the proof of  Lemma 5 in \cite{MAM08}  remain valid for the case where the lower bound is any affine function of the variable $z$ as we have. We therefore omit the proof and refer to that of   Lemma 5 in \cite{MAM08}.
\end{proof}

\begin{Theorem}
    The sequence $(Y^{m},Z^{m},U^{m})_{m\geq1}$ converges in $S^2\times L^2(W)\times L^2\tilde(N)$ to $(Y^{},Z^{},U^{})\in S^\infty\times L^2(W)\times L^2(\tilde{N})$ satisfying
    \beq
    Y_t & = & F + \int_{t}^{T} f_{}( Z_s, U_s) \, ds - \int_{t}^{T} Z_s \, dW_s - \int_{t}^{T} \int_{\R} U_s(e) \, \tilde{N}(ds, de), \quad t \in [0, T].   
    \enq
   % where the function $f_n$ is defined by
   % \beqs
   % f_n(z,u) & = & \inf_{n\geq p \geq -\underline{\pi}} f^{1}%(z,u,p)
%\label{defDriverfm}
% + \int_{\gamma(\R) \setminus \{0\}} \inf_{n\geq p \geq -%\underline{\pi}} f^{2}(z,u,g,p)
%\mu(dg) \\
%& & 
% - z \lambda C_{\kappa, \eta, K} - \frac{C_{\kappa, \eta, K}^2}{2 \lambda}
%\enqs
%for $(z,u)\in \R\times L^2(\tilde N)$. 
\end{Theorem}
\begin{proof} To complete the proof, we apply Proposition \ref{stabilityBSDE}. From  Lemmata \ref{LemQuadGrowfm} and  \ref{LemLipfm} and Proposition \ref{Prop ymbdd},  it only remains to show that the sequence $(f_{m})_{m\geq1}$ satisfies condition (i) of Proposition \ref{stabilityBSDE}. Fix a sequence $(z^m,u^m)_{m\geq1}$ of $\R\times L^2(\nu)$ converging to some $(z,u)\in\R\times L^2(\nu)$ with $(u^m)_{m\geq1}$ bounded in $L^\infty(\nu)$. Then we have
\beqs
    |f_{m}(z^m,u^m) -f(z,u)| & \leq &  |f_{m}(z^m,u^m) -f(z^m,u^m)| + |f(z^m,u^m) -f(z,u)| \;.
\enqs
Since the sequence $(z^m,u^m)_m$ is bounded in $\R\times L^\infty(\nu)$ we get from the definition of $f^m$ that
\beqs%\label{estim1stab}
    |f_{m}(z^m,u^m) -f(z^m,u^m)| & = & 0 
\enqs
for $m$ large enough. Then, %We turn to the term $|f(z^m,u^m) -f(z,u)|$. 
from Lemma \ref{lem-ref-f} we get  
\beqs
    |f(z^m,u^m)-f(z,u)| &  \xrightarrow[m\rightarrow+\infty]{} & 0\;,
\enqs
which gives the result.
\end{proof}

\section{Numerical illustration}\label{sect6}
%Let us recall that, in our framework, we consider $E = \mathbb{R}$ equipped with the $\sigma$-field $\mathcal{B}(\mathbb{R})$.\\
For the numerical tests, we consider consider a derivative of the form $F=g(S_T)$ where $g$ is a bounded Borel function and  a Poisson random measure with compensator
\beqs
\nu(de) & = & \rho|e|^{-\alpha}\,de,
\enqs
for some parameter $\alpha$. We choose $\alpha$ in such a way that the integrability condition on the L\'evy measure $\nu$ is satisfied and direct computations yield that $\alpha \in (1, 3)$. 
We fix $\eps\in(0,1)$ and we choose the function $\eta$ as follows \begin{equation*}
\eta(e) = \left\{
\begin{array}{rcl}
1-\eps & \mbox{ if } & e > 1-\eps \\
e & \mbox{ if } & |e| \leq 1-\eps \\
-1+\eps & \mbox{ if } & e <-1+\eps.
\end{array}
\right.
\end{equation*} Such an $\eta$ allows in particular to keep $S_t$ positive for all $t\in[0,T]$ if we start from $S_0=s_0>0$.

The condition   \eqref{condetanu1} is clearly satisfied.  As we assume ~\eqref{condetanu2}, we have to restrict $\alpha \in (1, 2).$

\subsection{Cases where we do not observe  small jumps.}
We define the signal function $\gamma$ by
\beqs
\gamma(e) & = & \eta(e) \, \mathds{1}_{\{ e \notin (-c, c) \}},\quad e\in\R\;.
\enqs
%where $\mathds{1}$ denotes the indicator function.\\
For the signal $\gamma$, we introduce a parameter $c \in (0, +\infty)$ such that we have no information on the size of jump of the underlying on the set $(-c, c)$. %In this way, when $c \rightarrow +\infty$, the signal reveals all the information about the jump i.e the investor know exactly what the stock price jump is about to happen and when $c \rightarrow +\infty$, all the information is hidden. 
%For the signal $\gamma$, we introduce a parameter $c \in (0, +\infty)$ such that we have no information on the size of jump of the underlying on the set $(-c, c)$. 
In this way, when $c \rightarrow +\infty$, the signal reveals all the information about the jump i.e the investor knows exactly what the stock price jump is about to happen and when $c \rightarrow 0 $, all the information is hidden.
%(see Figure~\ref{fig:eta_plot}\textcolor{red}{corriger la figure pour qu'elle ne touche pas 1 et -1}). 

%\begin{figure}[h!]
%    \centering
%    \includegraphics[width=0.6\textwidth]{téléchargement.png}
%    \caption{Shaded small jump region}
%    \label{fig:eta_plot}
%\end{figure}

\paragraph{Finding the measure $\mu$ on $\gamma \neq 0$.} For a Borel nonnegative function $f$ such that $f(0)=0$, we have
\beqs
\int_\R f(g)\mu(dg) & = & \int_\R f(\gamma(e))\nu(de)~~ =~~\int_{\mathbb{R}\setminus{(-c,c)} }f(\gamma(e))\nu(de) \;.
\enqs
If $c\geq 1-\eps$ we get
\beqs
\int_\R f(g)\mu(dg) & = &\rho\frac{c^{1-\alpha}}{\alpha - 1}(f(-1+\eps) + f(1-\eps)) \;. 
\enqs
If $c< 1-\eps$ we get
%$$ 
%\int_{\mathbb{R}\setminus{(-c,c)} }f(\gamma(e))\nu(de) = \int_{-\infty}^{-c}f(\eta(e))\nu(de) + \int_{c}^{\infty}f(\eta(e))\nu(de).$$ 
%Direct computations give us 
\beqs
\int_{\mathbb{R}} f(g)\mu(dg) & = & \rho\frac{(1-\eps)^{1-\alpha}}{\alpha - 1}(f(-1+\eps) + f(1-\eps)) + \int_{-1+\eps}^{-c}f(e) \nu(de) + \int_{c}^{1-\eps}f(e) \nu(de).
\enqs
We finally have  
\beqs
\mu(B) = \rho\frac{\big(c\vee(1-\eps)\big)^{1-\alpha}}{\alpha - 1}(\mathds{1}_B(-1+\eps) + \mathds{1}_B(1-\eps)) + \mathds{1}_{c<1-\eps}\Big(\int_{-1+\eps}^{-c}\mathds{1}_B(e) \nu(de) + \int_{c}^{1+\eps}\mathds{1}_B(e) \nu(de)\Big)
\enqs
for every  Borel set $B \subset \mathbb{R}\setminus{\{0\}}$. 
%\color{red} Have to mutiply the first term by $(1-\epsilon)^{1-\alpha}.

%\color{black}
\paragraph{Disintegration of the measure $\nu$ on $\gamma \neq 0$.}
From the definition of the function $\gamma$, we have $\{\gamma \neq 0\} = (-\infty, -c] \cup [c, \infty)$ and 
$\gamma(\R) \setminus \{0\} = [-1+\eps, -(c\wedge (1-\eps))] \cup [c\wedge (1-\eps), 1-\eps]$. We observe that in the case $c\geq 1-\eps$, we have $\gamma(\R)\setminus\{0\}=\{-(1-\eps),1-\eps\}$.\\

From the disintegration formula, the kernel $K$ satisfies
\beqs
\nu(de \cap \{\gamma \neq 0\}) 
&= & \nu(de)\, \mathds{1}_{(-\infty, -c] \cup [c, \infty)}(e) \\
&= & \int_{[-1+\eps, -(c\wedge (1-\eps))] \cup [c\wedge (1-\eps), 1-\eps]} K(g, de)\, \mu(dg) \;.\enqs
If $c\geq 1-\eps$ we get
\beqs
\nu(de \cap \{\gamma \neq 0\}) 
&= & \rho\frac{c^{1-\alpha}}{\alpha - 1} \Big(K(-1+\eps, de)+ K(1-\eps, de)\Big)\;.
\enqs
If $c< 1-\eps$ we get
\beqs
\nu(de \cap \{\gamma \neq 0\})  &= & \rho\frac{(1-\eps)^{1-\alpha}}{\alpha - 1}K(-1+\eps, de) 
   + \rho\int_{-1+\eps}^{-c} K(g, de)\, |g|^{-\alpha} \, dg \\
& &+ \rho\frac{(1-\eps)^{1-\alpha}}{\alpha - 1}K(1-\eps, de) 
   + \rho\int_{c}^{1-\eps} K(g, de)\, |g|^{-\alpha} \, dg.
\enqs
Using the condition $K(g,\R)=K(g,\gamma=g)$ we deduce that
\beqs
K(1-\eps,de) & = & \frac{\alpha - 1}{(c\vee (1-\eps))^{1-\alpha}}|e|^{-\alpha}\mathds{1}_{[c\vee (1-\eps),+\infty)}(e)de \\
K(-1+\eps,de) & = & \frac{\alpha - 1}{(c\vee (1-\eps))^{1-\alpha}}|e|^{-\alpha}\mathds{1}_{(-\infty,-(c\vee (1-\eps))]}(e)de
\enqs
and $K(g,de)=\delta_{\{g\}}(de)$ for $g\in (-(1-\eps),-c]\cup[c,1-\eps)$ in the case $c<1-\eps$. 
We are now able to compute the driver $f$ and the related optimal strategy.

\vspace{2mm}

Without any signal i.e. $\gamma=0$, the optimal strategy is given by
\beqs
p^*(0, z, u)
& = & \arg\min_{\hspace{-5.5mm}p \in [-\underline{\pi}, \overline{\pi}]}
\left\{
\frac{\lambda}{2}\left(p \sigma - \left(z + \frac{C_{\kappa,\eta}}{\lambda}\right)\right)^2
+ \int_{-c}^{c}
\left(
\frac{e^{\lambda (u(e) - p \eta(e))} - 1}{\lambda}
- u(e)
\right)\nu(de)
\right\}
\enqs
for $z \in \mathbb{R}$ and $u\in L^2(\nu)$. 
 
When receiving a signal $g \in \gamma(\mathbb{R})\setminus \{0\}$ we distinguish the three cases.
\begin{itemize}
%\color{red}
\item Case 1: $g = -1+\eps$. We have $\{\gamma=-1+\eps\}=(-\infty,-((1-\eps)\vee c)]$ and we get from the expression of $K$
\beqs
f^2(-1+\eps, u, p) & = & \frac{\alpha-1}{((1-\eps)\vee c)^{1-\alpha}}\int_{-\infty}^{-((1-\eps)\vee c)}[\frac{e^{\lambda (u(e) - p (-1+\eps))} - 1}{\lambda}
- u(e)]|e|^{-\alpha}de.
\enqs
In this case the optimal strategy is given by $p^*(-1+\eps)=-\underline\pi$ and
\beqs
\inf_{p\in[-\underline \pi,\overline \pi]}f^2(-1+\eps,u,p)& = & \frac{\alpha-1}{((1-\eps)\vee c)^{1-\alpha}}\int_{-\infty}^{-((1-\eps)\vee c)}[\frac{e^{\lambda (u(e) +\underline \pi (-1+\eps))} - 1}{\lambda}
- u(e)]|e|^{-\alpha}de.
\enqs
%\color{black}
%$$f^2(g, u, p) = \frac{1}{2c^{1-\alpha}}[\frac{e^{\lambda (u(-1) - p \eta(-1))} - 1}{\lambda}
%- u(-1)].$$
%\color{red}
\item Case 2: $g = 1-\eps$. We have $\{\gamma=1-\eps\}=[(1-\eps)\vee c,+\infty)$ and we get from the expression of $K$ 
\beqs
f^2(1-\eps, u, p) & = & \frac{\alpha-1}{((1-\eps)\vee c)^{1-\alpha}}\int_{(1-\eps)\vee c}^{+\infty} [\frac{e^{\lambda (u(e) - p (1-\eps))} - 1}{\lambda}
- u(e)]e^{-\alpha}de.
\enqs
In this case the optimal strategy is given by $p^*(1-\eps)=\overline \pi$ and
\beqs
\inf_{p\in[-\underline \pi,\overline \pi]}f^2(1-\eps,u,p)& = & \frac{\alpha-1}{((1-\eps)\vee c)^{1-\alpha}}\int_{(1-\eps)\vee c}^{+\infty}[\frac{e^{\lambda (u(e) -\overline \pi (1-\eps))} - 1}{\lambda}
- u(e)]e^{-\alpha}de.
\enqs

\item Case 3:  $c<1-\eps$ and $g\in (-1+\eps,-c]\cup[c,1-\eps)$.  We have $\{\gamma=g\}=\{g\}$ and we get from the expression of $K$ 
\beqs
f^2(g, u, p) & = & \frac{e^{\lambda (u(g) -p g)} - 1}{\lambda}
- u(g)\;,
\enqs
and
\beqs
\inf_{p\in[-\underline \pi,\overline \pi]}f^2(g, u, p) & = & \frac{e^{\lambda (u(g) +\underline \pi g)} - 1}{\lambda}\mathds{1}_{g<0}+\frac{e^{\lambda (u(g) -\overline \pi g)} - 1}{\lambda}\mathds{1}_{g>0}
- u(g)\;,
\enqs
\end{itemize}

%\color{red}
We finally get
\beq\label{Driverwithoutsmalljumps}%\nonumber
f(z,u) & = & \min_{p \in [-\underline{\pi}, \overline{\pi}]}
\left\{
\frac{\lambda}{2}\left(p\sigma  - \left(z + \frac{C_{\kappa,\eta}}{\lambda}\right)\right)^2
+ \int_{-c}^{c}
\left(
\frac{e^{\lambda (u(e) - p \eta(e))} - 1}{\lambda}
- u(e)
\right)\nu(de)
\right\}
\\
 &  &
 +
 \int_{ c}^{+\infty}[\frac{e^{\lambda (u(e) -\overline \pi ((1-\eps)\wedge e)} - 1}{\lambda}
- u(e)]\nu(de) \nonumber\\
 & & +%\frac{(1-\eps)^{1-\alpha}}{2c^{1-\alpha}}
 \int_{-\infty}^{- c }[\frac{e^{\lambda (u(e) +\underline \pi ((-1+\eps)\vee e)} - 1}{\lambda}
- u(e)]\nu(de)%\nonumber %\\
 %&  &
 - C_{\kappa,\eta} z-\frac{C_{\kappa,\eta}^2}{2\lambda} \nonumber
\enq
for $z \in \mathbb{R}$ and $u\in L^2(\nu)$. 

\color{black}
%To sum up,
%\[
%\begin{aligned}
%\int_{\gamma(\mathbb{R}) \setminus \{0\}} 
%\inf_{p \in [-\underline{\pi}, \overline{\pi}]} f^2(g, u, p)\,
%\mu(dg)
%&= \frac{1}{\alpha - 1}
%\left[ f^2(1, u(1), \overline{\pi}) + f^2(-1, u(-1), -\underline{\pi}) \right]  + \int_{-1}^{-c} 
%\left( 
%\frac{e^{\lambda (u(e) - p^*(u(e))e} - 1}{\lambda}
%- u(e)
%\right) \nu(de) \\
%&\quad + \int_{c}^{1} 
%\left(
%\frac{e^{\lambda (u(e) + p^*(u(e)) e} - 1}{\lambda}
%- u(e)
%\right) \nu(de).
%\end{aligned}

\subsection{Cases where we do not observe large jumps.}
In this part, we consider the case where large  jumps are not observed. The signal is then given by 
%In the first case, i.e.\ when we do not observe jumps above a certain size, we have
\beqs
\gamma(e) & = & \eta(e)\,\mathds{1}_{|e| \notin (c,+\infty)}\;,\quad e\in \R\;,
\enqs
where $c$ is a nonegative  constant.

\paragraph{Finding measure $\mu$ on $\{\gamma\neq0\}$.} For a Borel nonnegative function $f$ such that $f(0)=0$, we have
\beqs
\int_\R f(g)\mu(dg) & = & \int_\R f(\gamma(e))\nu(de)~~=~~\int_{[-c,0)\cup(0,c]}f(\gamma(e))\nu(de)\;.
\enqs
If $c\geq 1-\eps$, we get
\beqs
\int_\R f(g)\mu(dg) &  = &  \int_{-c}^{-1 + \eps}f(-1 + \eps)\nu(de) + \int_{-1 + \eps}^{1-\eps}f(e)\nu(de) + \int_{1 - \eps}^{c}f(1 - \eps)\nu(de)\\
& = & \rho\frac{(1-\eps)^{1-\alpha}-c^{1-\alpha}}{\alpha-1}\big(f(-1 + \eps)+f(1 - \eps)\big) + \int_{-1 + \eps}^{1-\eps}f(e)\nu(de).
\enqs 
If $c<1-\eps$, we get 
\beqs
\int_\R f(g)\mu(dg) &  = &  \int_{-c}^{c}f(e)\nu(de) ~~=~~ \int_{-c}^{c}f(e)\nu(de).
\enqs 
In particular, for every Borel set $B \subset [-c,c]\setminus \{0\}$, we have
\beqs
\mu(B)
 & = & 
\rho\frac{(1-\eps)^{1-\alpha}-(c\vee(1-\eps))^{1-\alpha}}{\alpha-1}\big(
\mathds{1}_{B}(-1+\eps)
+\mathds{1}_{B}(1-\eps)\big)
+
\int_{-((1-\eps)\wedge c)}^{(1-\eps)\wedge c} \mathds{1}_{B}(e)\,\nu(de).
\enqs
\paragraph{Disintegration of the measure $\nu$ on $\{\gamma \neq 0\}$.}
From the definition of $\gamma$, we have
$\{\gamma \neq 0\}  =  [-c,0)\cup(0, c]$ and   
$\gamma(\R)\setminus\{0\} = [-(c\wedge(1-\eps)),0)\cup(0,c\wedge(1-\eps)]$.
From the disintegration formula, the kernel $K$ satisfies
\beqs
\nu(de \cap \{\gamma \neq 0\}) 
&= & \nu(de)\, \mathds{1}_{[-c,0)\cup(0, c]}(e) \\
 & = & \int_{[-(c\wedge(1-\eps)),0)\cup(0,c\wedge(1-\eps)]} K(g,de)\mu(dg)\;. \enqs
If $c\geq 1-\eps$ we get
 \beqs
\nu(de \cap \{\gamma \neq 0\}) &= & \rho\frac{(1-\eps)^{1-\alpha}-c^{1-\alpha}}{\alpha-1}\big(
K(-1+\eps, de)
+K(-1+\eps, de)\big) 
   + \rho\int_{-1+\eps}^{1-\eps} K(g, de)\, |g|^{-\alpha} dg.
\enqs
If $c< 1-\eps$ we get
 \beqs
\nu(de \cap \{\gamma \neq 0\}) &= &  \rho\int_{-c}^{c} K(g, de)\, |g|^{-\alpha} dg.
\enqs
Using the condition $K(g,\R)=K(g,\gamma=g)$ we deduce that $K(g,de)=\delta_{\{g\}}(de)$ for $g\in (-((1-\eps)\wedge c);c\wedge(1-\eps))$ and
\beqs
K(1-\eps,de) & = & \frac{\alpha - 1}{(1-\eps)^{1-\alpha}-(c\vee (1-\eps))^{1-\alpha}}|e|^\alpha\mathds{1}_{[1-\eps,c\vee(1-\eps))}(e)de \\
K(-1+\eps,de) & = & \frac{\alpha - 1}{(1-\eps)^{1-\alpha}-(c\vee (1-\eps))^{1-\alpha}}|e|^\alpha\mathds{1}_{(-(c\vee (1-\eps)),-(1-\eps)]}(e)de
\enqs
with the convention $\infty.0=0$. 
We are now able to compute the driver $f$ and the related optimal strategy.

The optimal strategy without signal (i.e., $\Delta_t G = 0$)  is given by
\beqs
p^*(0, z, u)
& = & \arg\min_{\hspace{-5.5mm}p \in [-\underline{\pi}, \overline{\pi}]}
\left\{
\frac{\lambda}{2}\left(p\sigma  - \left(z + \frac{C_{\kappa,\eta}}{\lambda}\right)\right)^2
+ \int_{(-\infty,-c)\cup(c,+\infty)}
\left(
\frac{e^{\lambda (u(e) - p \eta(e))} - 1}{\lambda}
- u(e)
\right)\nu(de)
\right\}
\enqs
for $z \in \mathbb{R}$ and $u\in L^2(\nu)$. 

When receiving a signal $g \in \gamma(\mathbb{R})\setminus \{0\}$ we distinguish four cases.
\begin{itemize}
%\color{red}
\item Case 1: $g = -1+\eps$ and $c\geq 1-\eps$. We have $\{\gamma=-1+\eps\}=[-c,-1+\eps]$ and we get from the expression of $K$
\beqs
f^2(-1+\eps, u, p) & = & \frac{\alpha - 1}{(1-\eps)^{1-\alpha}-(c\vee (1-\eps))^{1-\alpha}} \int_{-c}^{-1+\eps}[\frac{e^{\lambda (u(e) - p (-1+\eps))} - 1}{\lambda}
- u(e)]|e|^{-\alpha}de.
\enqs
In this case the optimal strategy is given by $p^*(-1+\eps)=-\underline\pi$ and
\beqs
\inf_{p\in[-\underline \pi,\overline \pi]}f^2(-1+\eps,u,p)& = & \frac{\alpha - 1}{(1-\eps)^{1-\alpha}-(c\vee (1-\eps))^{1-\alpha}} \int_{-c}^{-1+\eps}[\frac{e^{\lambda (u(e) +\underline \pi (-1+\eps))} - 1}{\lambda}
- u(e)]|e|^{-\alpha}de.
\enqs

\item Case 2: $g = 1-\eps$ and $c\geq 1-\eps$. We have $\{\gamma=1-\eps\}=[1-\eps,c]$ and we get from the expression of $K$
\beqs
f^2(1-\eps, u, p) & = & \frac{\alpha - 1}{(1-\eps)^{1-\alpha}-(c\vee (1-\eps))^{1-\alpha}}\int_{1-\eps}^c[\frac{e^{\lambda (u(e) - p (1-\eps))} - 1}{\lambda}
- u(e)]|e|^{-\alpha}de.
\enqs
In this case the optimal strategy is given by $p^*(1-\eps)=\overline\pi$ and
\beqs
\inf_{p\in[-\underline \pi,\overline \pi]}f^2(1-\eps,u,p)& = & \frac{\alpha - 1}{(1-\eps)^{1-\alpha}-(c\vee (1-\eps))^{1-\alpha}} \int_{1-\eps}^c[\frac{e^{\lambda (u(e) +\overline \pi (-1+\eps))} - 1}{\lambda}
- u(e)]|e|^{-\alpha}de.
\enqs

\item Case 3: $g \in  (-1+\eps,1-\eps)$ and $c\geq 1-\eps$.  We have $\{\gamma=g\}=\{g\}$ and we get from the expression of $K$ 
\beqs
f^2(g, u, p) & = & \frac{e^{\lambda (u(g) -p g)} - 1}{\lambda}
- u(g)\;,
\enqs
and
\beqs
\inf_{p\in[-\underline \pi,\overline \pi]}f^2(g, u, p) & = & \frac{e^{\lambda (u(g) +\underline \pi g)} - 1}{\lambda}\mathds{1}_{g<0}+\frac{e^{\lambda (u(g) -\overline \pi g)} - 1}{\lambda}\mathds{1}_{g>0}
- u(g)\;,
\enqs
\item Case 4: $g \in  (-c,c)$ and $c< 1-\eps$.  We have $\{\gamma=g\}=\{g\}$ and we get from the expression of $K$ 
\beqs
f^2(g, u, p) & = & \frac{e^{\lambda (u(g) -p g)} - 1}{\lambda}
- u(g)\;,
\enqs
and
\beqs
\inf_{p\in[-\underline \pi,\overline \pi]}f^2(g, u, p) & = & \frac{e^{\lambda (u(g) +\underline \pi g)} - 1}{\lambda}\mathds{1}_{g<0}+\frac{e^{\lambda (u(g) -\overline \pi g)} - 1}{\lambda}\mathds{1}_{g>0}
- u(g)\;,
\enqs

\end{itemize}

We finally get the following expression for the driver 
\beq\nonumber
f(z,u) & = & \min_{p \in [-\underline{\pi}, \overline{\pi}]}
\left\{
\frac{\lambda}{2}\left(p \sigma - \left(z + \frac{C_{\kappa,\eta}}{\lambda}\right)\right)^2
+ \int_{(-\infty,-c)\cup(c,+\infty)}
\left(
\frac{e^{\lambda (u(e) - p \eta(e))} - 1}{\lambda}
- u(e)
\right)\nu(de)
\right\}\\\nonumber
 &  &  +\int_{-c}^{0}[\frac{e^{\lambda (u(e) +\underline \pi (e\vee(-1+\eps))} - 1}{\lambda}
- u(e)]\nu(de) %\\
% &  &
+\int_{0}^c[\frac{e^{\lambda (u(e) +\overline \pi (e\wedge (1-\eps))} - 1}{\lambda}
- u(e)]\nu(de)\\
 &  & - C_{\kappa,\eta} z-\frac{C_{\kappa,\eta}^2}{2\lambda} \label{Driverwithoutbigjumps}
\enq
for $z \in \mathbb{R}$ and $u\in L^2(\nu)$. 
 \paragraph{Approximation of the driver $f$. } We use the sequence $(e_i)_{i = -(q+1)}^{q+1}$ of the previous section and 
we introduce the integers $\tilde \ell$ and $\tilde m$ such that
\beqs
e_{\tilde \ell}~<~1-\eps~\leq~e_{\tilde \ell+1} & \mbox{ and } & e_{\tilde \ell+\tilde m}~<c\leq e_{\tilde \ell+\tilde m+1}\;. 
\enqs
We then approximate the driver $f$ by the function $\check f$ where the integrals are discretized according to the measure $\check \nu$. More precisely, for instance for \eqref{Driverwithoutbigjumps}, $\tilde  f$ is given by
\beqs
\tilde f(z,u) & = & \min_{p \in [-\underline{\pi}, \overline{\pi}]}
\left\{
\frac{\lambda}{2}\left(p\sigma  - \left(z + \frac{C_{\kappa,\eta}}{\lambda}\right)\right)^2
+ \sum_{-\tilde q\leq i\leq \tilde q\;,~ |i|>\tilde \ell+\tilde m}
\left(
\frac{e^{\lambda (u(e_i) - p e_i)} - 1}{\lambda}
- u(e_i)
\right)\nu_i
\right\}\\
 &  & +\frac{1}{2c^{1-\alpha}} \sum_{i=1 }^{\tilde \ell}[\frac{e^{\lambda (u(e_i) -\overline \pi (1-\eps))} - 1}{\lambda}
- u(e_i)]\nu_i\\
 & & +\frac{1}{2c^{1-\alpha}}\sum_{i=-\tilde \ell}^{-1}[\frac{e^{\lambda (u(e_i) +\underline \pi (-1+\eps))} - 1}{\lambda}
- u(e_i)]\nu_i - C_{\kappa,\eta} z-\frac{C_{\kappa,\eta}^2}{2\lambda}.
\enqs
%RESTE A ECRIRE L4EDSR AVEC DISCRETIZATION EN ESPCE DES SAUTS
\color{black}

\subsection{Jump space discretization and numerical schemes for BSDEs}
To be able to set  numerical schemes for the approximation of solution to BSDEs with the previusly computed drivers,  we need to discretize the jump measure in space.% to use numerical methods to compute theses integral which don't have an explicit formula and depend of the control $z, u$.

\paragraph{Space discretization of the Poisson measure.}
We introduce an increasing sequence \((e_i)_{i = -(q+1)}^{q+1}\) 
such that $e_{q+1} = -e_{-(q+1)} = +\infty$ and $e_{-i}=-e_i$ for all $i=0,\ldots,q$. In particular, we have  $e_0 = 0$.
%We can write
%\[
%\mathbb{R} = \bigcup_{i = -q}^{q } [e_i, e_{i + 1}],
%\]
%\textcolor{red}{pb fermeture intervalles extremaux}
%and 
Let us introduce the  weights 
\[
\nu_i = \int_{\frac{e_{i - 1} + e_{i}}{2}}^{\frac{e_{i + 1} + e_i}{2}} \nu(de)\;,\quad -q,\ldots,-1,1,\ldots,q\;.
\]
We next write the random measure $N$ as follows
\beqs
N & = & \sum_{k \geq 0} \delta_{(T_k, \zeta_k)}
\enqs
where $(T_k)_{k\geq 0}$ is a nondecreasing sequence of stopping times and $\zeta_k$ is an $\Fc_{T_k}$-measurable random variable.
We then approximate $N$ by the random measure $\check N$ defined by 
\beqs
\sum_{k \geq 0} \delta_{(T_k, \check{\zeta}_k)},
\enqs
with
\beqs
\check{\zeta}_k = \sum_{i = -q,\;i\neq0}^{q} e_i \,
\mathds{1}_{\left( \frac{e_{i - 1} + e_{i }}{2}, \, \frac{e_{i + 1} + e_i}{2} \right]}(\zeta_k).
\enqs
From classical results on Poisson measures, $\check N$ is a Poisson measure  on  $\R_+\times \{e_{-q},\ldots,e_{-1},e_1,\ldots,e_q$\} with compensator $\check\nu(de)dt$ where $\check \nu$ is given by
\beqs
\check \nu(de) & = & \sum_{i = -q,\;i\neq0}^{q}\nu_i\delta_{e_i}(de)\;.
\enqs
 \paragraph{Approximation of the drivers. } 
We introduce the integer $\ell$ such that
\beqs
e_\ell & < & c~\leq~e_{\ell+1} %& \mbox{ and } & e_{\ell+m}~<1-\eps\leq e_{\ell+m+1}
\;. 
\enqs
We then approximate the driver $f$ by the function $\check f$ where the integrals are discretized according to the measure $\check \nu$. 
\begin{enumerate}
    \item \underline{Case of a signal given by large jumps}. Using \reff{Driverwithoutsmalljumps}, $\check f$ is in this case given by
\beqs
\check f(z,u) & = & \min_{p \in [-\underline{\pi}, \overline{\pi}]}
\left\{
\frac{\lambda}{2}\left(p\sigma  - \left(z + \frac{C_{\kappa,\eta}}{\lambda}\right)\right)^2
+ \sum_{i=-\ell, i\neq 0}^{\ell}
\left(
\frac{e^{\lambda (u(e_i) - p e_i)} - 1}{\lambda}
- u(e_i)
\right)\nu_i
\right\}\\
 &  & +%\frac{1}{2c^{1-\alpha}} 
 \sum_{i=\ell+1}^{q}[\frac{e^{\lambda (u(e_i) -\overline \pi ((1-\eps)\wedge e_i))} - 1}{\lambda}
- u(e_i)]\nu_i\\
 & & +%\frac{1}{2c^{1-\alpha}}
 \sum_{i=-q}^{-(\ell+1)}[\frac{e^{\lambda (u(e_i) +\underline \pi ((-1+\eps)\vee e_i)} - 1}{\lambda}
- u(e_i)]\nu_i\\
 &  & - C_{\kappa,\eta} z-\frac{C_{\kappa,\eta}^2}{2\lambda}.
\enqs

    \item \underline{Case of a signal given by small jumps}. Using \reff{Driverwithoutbigjumps}, $\check f$ is in this case given by
\beqs
\check f(z,u) & = & \min_{p \in [-\underline{\pi}, \overline{\pi}]}
\left\{
\frac{\lambda}{2}\left(p \sigma - \left(z + \frac{C_{\kappa,\eta}}{\lambda}\right)\right)^2
+ \sum_{|i|> \ell }
\left(
\frac{e^{\lambda (u(e_i) - p e_i)} - 1}{\lambda}
- u(e_i)
\right)\nu_i
\right\}\\
 &  & +%\frac{1}{2c^{1-\alpha}} 
 \sum_{i=0}^{\ell}[\frac{e^{\lambda (u(e_i) -\overline \pi ((1-\eps)\wedge e_i))} - 1}{\lambda}
- u(e_i)]\nu_i\\
 & & +%\frac{1}{2c^{1-\alpha}}
 \sum_{i=-\ell}^{-1}[\frac{e^{\lambda (u(e_i) +\underline \pi ((-1+\eps)\vee e_i)} - 1}{\lambda}
- u(e_i)]\nu_i\\
 &  & - C_{\kappa,\eta} z-\frac{C_{\kappa,\eta}^2}{2\lambda}.
\enqs

    \item \underline{Case without signal}. In this case, we get from \cite{MAM08} the following expression for the driver
    \beqs\nonumber
\hspace{-8mm}f(z,u)  =  \min_{p \in [-\underline{\pi}, \overline{\pi}]}
\left\{
\frac{\lambda}{2}\left(p\sigma  - \left(z + \frac{C_{\kappa,\eta}}{\lambda}\right)\right)^2
+ \int_{\R}
\left(
\frac{e^{\lambda (u(e) - p \eta(e))} - 1}{\lambda}
- u(e)
\right)\nu(de)
\right\} - C_{\kappa,\eta} z-\frac{C_{\kappa,\eta}^2}{2\lambda},\\\nonumber
% &  &  +\int_{-c}^{0}[\frac{e^{\lambda (u(e) +\underline \pi (e\vee(-1+\eps))} - 1}{\lambda}
%- u(e)]\nu(de) %\\
% &  &
%+\int_{0}^c[\frac{e^{\lambda (u(e) +\overline \pi (e\wedge (1-\eps))} - 1}{\lambda}
%- u(e)]\nu(de)\\
% &  & -\lambda C_{\kappa,\eta} z-\frac{C_{\kappa,\eta}^2}{2\lambda} \label{Driverwithoutbigjumps}
\enqs
for $z \in \mathbb{R}$ and $u\in L^2(\nu)$. Therefore $\check f$ is given by
\beqs
\hspace{-8mm}f(z,u) = \min_{p \in [-\underline{\pi}, \overline{\pi}]}
\left\{
\frac{\lambda}{2}\left(p\sigma  - \left(z + \frac{C_{\kappa,\eta}}{\lambda}\right)\right)^2
+ \sum_{i=- q }^q
\left(
\frac{e^{\lambda (u(e_i) - p e_i)} - 1}{\lambda}
- u(e_i)
\right)\nu_i
\right\} - C_{\kappa,\eta} z-\frac{C_{\kappa,\eta}^2}{2\lambda}\;.\enqs
\end{enumerate}

%\subsection{Approximation of the driver in the classical case without signal.}
%Due to the lack of closed formula, we consider the solution without signal treated by Morlais in \cite{MAM08} to observe asymptotic property of the value function with respect to $c$. \\
%In her case, the driver is defined by 
%\begin{equation}
%    f(z, u) = \min_{p \in [-\underline{\pi}, \overline{\pi}]}
%\left\{ ( \frac{\lambda}{2} |p\sigma - (z + \frac{\theta}{\lambda})|^2 + \int_E h(u(e) - p\eta(e)) \nu(de)\right\} - \theta z - \frac{\theta^2}{2 \lambda}
%\end{equation} 
%with the premium risk $$\theta = \frac{\kappa}{\sigma}.$$\\
%This driver can be approximate with the previous discretization of the space $E$ by the function $\check f$ as follow 

%\begin{equation*}
%    \check f(z, u) = \min_{p \in [-\underline{\pi}, \overline{\pi}]}
%\left\{ ( \frac{\lambda}{2} |p\sigma - (z + \frac{\theta}{\lambda})|^2 + \sum_{i = -q, i \neq0}^{q} h(u(e_i) - p\eta(e_i)) \nu_i \right\}- \theta z - \frac{\theta^2}{2 \lambda}.
%\end{equation*}
%For the simulation in order to avoid explosion of the driver we consider the driver with two truncatures functions,  one in the  $z$ component $(\rho_m)_m$ and another one $(\varphi_m)_m$ in the $u$ component. \\
%To solver the different BSDE with jumps, we implement the following algorithm :
\paragraph{Discrete-time scheme for the BSDEJ.}
Let us denote by $\check N(i)$ the process $(\check N([0,t]\times \{e_i\})_{t\ge0}$ for $i=-q,\ldots,-1,1\ldots,q$.
We follow the discretization scheme proposed by Bouchard and Elie \cite{bouchard2008discrete}.
For that, we fix a time grid $\{0 = t_0 < t_1 < \cdots < t_n = T\}$ and we define the approximation scheme $(\bar Y_{t_k},\bar Z_{t_k}, (\bar U_{t_k}(e_i))_{i\neq 0 })_{0\leq k\leq n}$  by 
\beqs
\bar Y_{t_n} = g(S_{t_n}),
\enqs
and for \(k = n-1, \ldots, 0\),
\beqs
 \Delta t_k  & = & t_{t_{k+1}} - t_{t_k}\, \\ 
 \Delta_k W & = & W_{t_{k+1}} - W_{t_k}\, \\ 
\bar{Z}_{t_k} &= &\frac{1}{\Delta t_k}\, \mathbb{E}\!\left[\bar{Y}_{t_{k+1}} \, \Delta_k W \,\big|\, S_{t_k}\right], \\
\Delta_k \tilde N(i) & = &  N_{t_{k+1}}(i) - N_{t_k}(i)-\nu_i\Delta t_k\,\\
\bar{U}_{t_k}(e_i) &= & \frac{1}{\nu_i\Delta t_k}\, 
\mathbb{E}\!\left[\bar{Y}_{t_{k+1}} \Delta_k \tilde N(i)\,\big|\, S_{t_k}\right], ~-q\leq i\leq q\;, i\neq 0\;,\\ 
\bar{Y}_{t_k} &= &  \mathbb{E}\!\left[\bar{Y}_{t_{k+1}} \,\big|\, S_{t_k}\right] 
+ \Delta t_k\, \check f\big(\bar{Z}_{t_k}, \bar{U}_{t_k}\big).
\enqs
We notice even the algorithm uses the true value of $S$, it remains computable as $S$ can be exactly simulated.\\
We compute estimators of the conditional expectations using least squares regressions on adaptative local basis functions, see \cite{bouchardmonte}.

%\paragraph{Algorithm for the quadratic BSDEJ.}
%Let us denote by $\check N(i)$ the process $(\check N([0,t]\times \{e_i\})_{t\ge0}$ for $i=-q,\ldots,-1,1\ldots,q$.
% We follow the discretization scheme proposed by B. and Elie \cite{bouchard2008discrete}.
%As in the previous section,  we fix a time grid $\{0 = t_0 < t_1 < \cdots < t_n = T\}$ and we define the approximation scheme $(\bar Y_{t_k},\bar Z_{t_k}, (\bar U_{t_k}(e_i))_{i\neq 0 })_{0\leq k\leq n}$  by 
%\beqs
%\bar Y_{t_n} = g(S_{t_n}),
%\enqs
%and for \(k = n-1, \ldots, 0\),
%\beqs
 %\Delta t_k  & = & t_{t_{k+1}} - t_{t_k}\, \\ 
 %\Delta_k W & = & W_{t_{k+1}} - W_{t_k}\, \\ 
%\bar{Z}_{t_k} &= &\frac{1}{\Delta t_k}\, \mathbb{E}\!\left[\bar{Y}_{t_{k+1}} \, \Delta_k W \,\big|\, S_{t_k}\right], \\
%\Delta_k \tilde N(i) & = &  N_{t_{k+1}}(i) - N_{t_k}(i)-\nu_i\Delta t_k\,\\
%\bar{U}_{t_k}(e_i) &= & \frac{1}{\nu_i\Delta t_k}\, 
%\mathbb{E}\!\left[\bar{Y}_{t_{k+1}} \Delta_k \tilde N(i)\,\big|\, S_{t_k}\right], ~-q\leq i\leq q\;, i\neq 0\;,\\ 
%\bar{Y}_{t_k} &= &  \mathbb{E}\!\left[\bar{Y}_{t_{k+1}} \,\big|\, S_{t_k}\right] 
%+ \Delta t_k\, \tilde f\big(\bar{Z}_{t_k}, \bar{U}_{t_k}\big).
%\enqs\\
%\\
\subsection{Numerical results}
We consider put options with maturity $T = 1$  year and a strike price equal to $s_0 = 1$ and the following parameters 
\begin{equation*}
   \rho = 0.1,\quad \alpha = 1.5,\quad \kappa = \int_\mathbb{R}\eta(e)\nu(de) = 0,\quad  \varepsilon =0.01,\quad \lambda = 0.4,\quad \sigma = 0.2,\quad \mbox{and }n=10 \mbox{ time steps}.  %(Define the intensity of the Poisson process and replace $\nu$ by the intensity times $\nu$). 
\end{equation*}
When we don't have signal, we are in the classical case and the simulated $Y_0$ is similar to the solution with the driver found in \cite{MAM08}.
As expected from the theoretical framework, the value function at $t = 0$ exhibits
monotonicity with respect to the signal parameter $c$. This property is mathematically related to the comparison theorem for BSDE with jumps: if the terminal condition or the driver increases, the solution at time zero increases accordingly. From a financial standpoint, this monotonicity has a clear and economically meaningful interpretation.
For example in the  small jumps hidden signal case when $c$ increases, $Y_0$ increases. This reflects the fact that a stronger signal provides to the investor a 
more precise information about the jump dynamics of the risky assets. Armed with this additional information, the investor can better anticipate discontinuities in asset prices, adjust his/her portfolio strategy more efficiently, and ultimately achieve a higher expected utility of terminal wealth. In other words, information has a quantifiable economic value in this framework, and $Y_0$ measures it. \\
\begin{figure}[htbp]
    \centering
    \begin{minipage}[b]{0.49\textwidth}
        \centering
        \includegraphics[width=\textwidth]{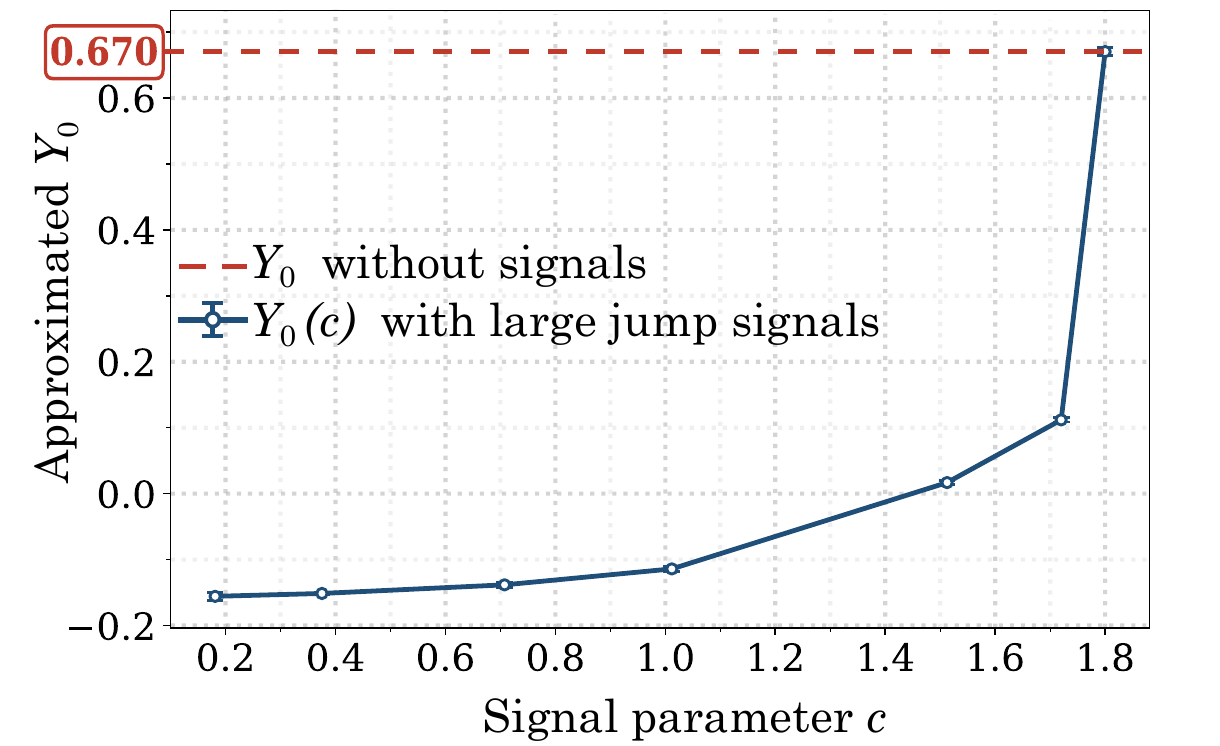}
       
    \end{minipage}
    \hfill
    \begin{minipage}[b]{0.495\textwidth}
        \centering
        \includegraphics[width=\textwidth]{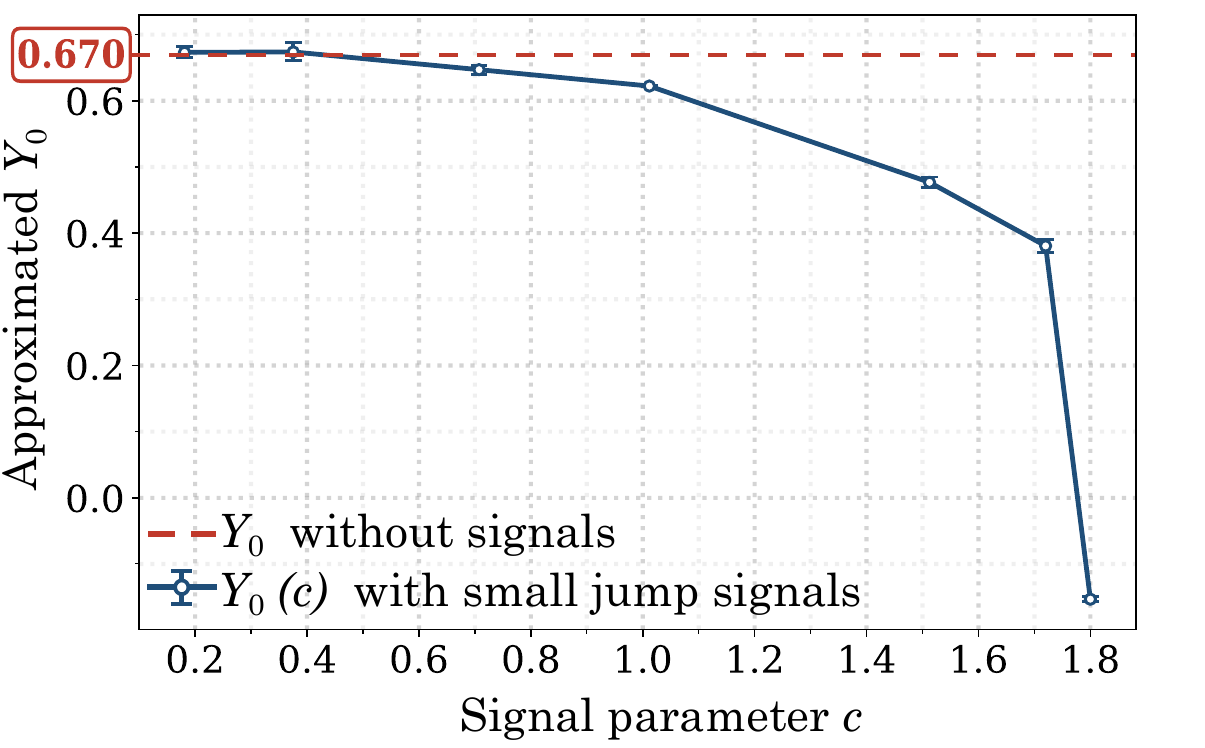}

    \end{minipage}
     \caption{The Monte Carlo simulation of $Y_0$: on the left for $\gamma(e) = \eta(e)\,\mathds{1}_{e \notin (-c,c)}\;$ and on the right for
     $\gamma(e) = \eta(e)\,\mathds{1}_{|e| \notin (c, +\infty)}\;$.}
        \label{fig:center}
\end{figure}
\begin{comment}
    
A particularly interesting situation arises when $Y_0$ becomes negative. 
To interpret this, we relate $Y_0$ to the indifference price of the claim $F$. 
When $Y_0 < 0$, the investor is effectively in a situation where the risk-adjusted cost becomes negative, meaning that the position improves their utility. This situation typically arises when the investor has sufficiently precise information about the jump structure of the asset. 
In that case, the residual risk induced by the jump component $\eta(e)$ can be efficiently managed, so that the overall position becomes beneficial from a utility perspective. 
As a result, the claim becomes financially attractive, even when accounting for risk aversion.

Importantly, this phenomenon is not related to the presence of arbitrage opportunities. 
The market remains arbitrage-free, and equivalent martingale measures exist. 
However, the jump component introduces a source of incompleteness that prevents perfect hedging. 
Improved information allows the investor to better control this source of risk, thereby reducing the associated certainty equivalent $Y_0$. 
The indifference price, which depends on both risk aversion and the available information, naturally reflects this improvement.

In summary, the decrease of $Y_0$ as the signal quality increases captures the economic value of information in incomplete markets with jumps: better information reduces effective uncertainty, lowers the risk-adjusted cost of holding the position, and enhances the investor’s welfare.
\end{comment}

\newpage{}
\appendix
\section{A comparison result for BSDEs with jumps}\label{sect7}

We provide a comparison result for BSDEs with jumps. This results is similar to Theorem 2.5 in \cite{Royer06}, but holds under a slightly different  assumption that are satisfied in the case we consider.\\
%First of all, we can remark that for all $m$, $$\mathbb{E}(\int_0^t|f_m(0,0)|^2 ds) \leq \infty$$ due to finiteless of the measure $\nu_m$, hypothesis 2.2 and boundaries of the set of valeurs of the strategies.
\begin{Theorem}\label{ComparisonBSDEJ}
Let $g^1, g^2$ be two functions from $ \R\times L^2(\nu)$ satisfying the following conditions. \begin{enumerate}[(i)]
\item $g^1$ is Lipchitz continuous w.r.t. its first variable: there exists a constant $L$ such that 
\beqs
|g^1(z,u)-g^1(z',u)| & \leq & L|z-z'|
\enqs
for all $z,z'\in \R$ and $u\in L^2(\nu)$.
\item  There exists a Borel function $k:~L^2(\nu)\times L^2(\nu)\times \R\rightarrow \R$ such that 
\beqs
     g^{1}(z,u) - g^{1}(z,u') & \leq &  \int_{\R}k(u, u')(e)(u(e) - u'(e))\nu(de)
 \enqs
 for all $ z \in \mathbb{R}$  and $u, u' \in  L^{\infty}(\nu)\cap L^{2}(\nu)$, 
%\begin{equation}
 %   |f(z, u) - f(z', u)|\leq C_3( \alpha + |z| + |z'|)|z-z'|
%\end{equation}
%for all $z, z' \in \mathbb{R}, u \in L^{\infty}(\nu)$.
\item For each $M>0$, there exist constants $\overline C_{M},\underline C_{M}>0$ and $\alpha>0$ such that the function $k$ satisfies
\beqs
\big(-1 + \underline{C}_{M}%e^{-\lambda (n+\underline{\pi})|\eta(e)|}
\big)\mathds{1}_{|e|>\alpha}  &\leq & k(u, u')(e) ~~\leq~~\mathds{1}_{|e|>\alpha}\overline C_{M}%\;,\quad e\in \R  \;,
\enqs
%and 
%\beqs
%\|k(u,u')\|_{L^2(\nu)} & \leq & \bar C_M
%\enqs
for all $u, u' \in  L^{2}(\nu)\cap L^{\infty}(\nu)$  such that  $\|u\|_{L^{\infty}(\nu)},\, \|u'\|_{L^{\infty}(\nu)} \leq M$ and all $e\in \R  $.
\end{enumerate}
Let $F^1$ and $F^2$ be two bounded $\mathcal{F}_T$-measurable random variables and $(Y^i,Z^i, U^i)\in S^\infty\times L^2(W)\times L^2(\tilde N)$ satisfying
\beqs
Y^i_t  &= & F^i+\int_t^Tg^i(Z^i_s,U^i_s)ds-\int_t^TZ^i_sdW_s-\int_t^T\int_\R U^i_s(e)\tilde N(de,ds)\;,\quad t\in[0,T]\;,
\enqs
for ${i = 1,2}$. 

If $F^2 \geq F^1$ $\P$-a.s. and $g^2 \geq g^1$ on $\R\times L^2(\nu)$, then 
\beqs
 Y^1_t & \leq & Y^2_t\;,\quad t\in[0,T]\;.
 \enqs
\end{Theorem}

\begin{proof}
    %The lipschitz property of theses generator garantees existence and uniqueness of the solution $(Y^i,Z^i,U^i)_{i = 1,2}$ of each BSDE.\\
    We define 
\beqs
    \hat{F} ~=~ F^2 - F^1,\;~
        \hat{g} ~=~ g^2 - g^1,\;~
    \hat{Y} ~=~ Y^2 - Y^1,\;~
    \hat{Z} ~=~  Z^2 - Z^1\;~
    \hat{U}~ =~ U^2 - U^1\;.
\enqs
We also define the process $\beta$ by
\begin{equation*}
\beta_t
   =   \left\{
 \begin{array}{cc}\frac{g^1(Z^2_t, U^2_t) - g^1(Z^1_t, U^2_t)}{ Z^2_t - Z^1_t } & \mbox{ if }
    Z^1_t \neq Z^2_t \\
     0 & \mbox{ else}
     \end{array}\right.
\end{equation*}
for $t\in[0,T]$. Since $g^1$ is Lipschitz w.r.t. $z$, $\beta$ is bounded.
%By linearizing the driver $\hat{g_m} = g^1_m - g^2_m$ 
We then  have  
\beqs
\hat{Y}_t & = & \hat{F} + \int_t^T \hat{g}(Z^2_s, U^2_s)ds + \int_t^T(g^1(Z^1_s, U^2_s) - g^1(Z^1_s,U^1_s))ds +\int_t^T \beta_s \hat{Z}_sds \\
 & & -\int_t^T \hat Z_s dW_s -\int_t^T \int_\R \hat{U}_s(e) \tilde{N}(ds, de)\;,\quad t\in[0,T]\;.
 \enqs
Using $\hat{F}\geq 0$, $\hat{g} \geq 0$ and \beqs
     g^1_{}(Z^1_s,U^2_s) - g^1_{}(Z^1_s,U^1_s) & \leq &  \int_{\R}k_{}(U^2_s, U^1_s)(e)(U^2_s(e)-U^1_s(e))\nu(de)
 \enqs
 we have 
 \beq \label{ineqdeltaY}
 \hat{Y}_t & \geq &  \int_t^T \left(\beta_s \hat{Z}_s + \int_\R k(U^2_s,U^1_s)(e) \hat{U_s}(e)\nu(de) \right) ds -\int_t^T \hat Z_s dW_s -\int_t^T \int_\R \hat{U}_s(e) \tilde{N}(ds, de) \enq
for all $t\in[0,T]$. We next define the process $\Ec_t$ as the unique solution to 
\beqs
\Ec_t & = & 1+\int_0^t \Ec_{s-}\Big( \beta_s dW_s+\int_\R k(U^2_s,U^1_s)(e) \tilde{N}(ds, de) \Big)\;,\quad t\in[0,T]
\enqs
Since $\beta$ and $k(U^2,U^1)$ are bounded and $k(U^2,U^1)$ satisfies $(iii)$, this SDE admits a unique solution which is a positive martingale and we have
\beqs
\sup_{t\in[0,T]} \E\Big[|\Ec_t|^2\Big] & < & +\infty\;.
\enqs
We can define the equivalent probability measure $\hat \P$ on $\Fc_T$ by
\beqs
\left.\frac{d\hat \P}{d\P}\right|_{\Fc_T} & = & \Ec_T
\enqs
By Girsanov Theorem,  $\hat W_t = W_t - \int_0^t \beta_s ds$ is a Brownian motion under $\hat \P$ and $N$ is an independent Poisson measure with compensator measure  $\hat \nu(de):= \nu(de)(1+k(U^1,U^2)(e))$ under $\hat \P$.
In particular, the process $(M_t)_{t\in[0,T]}$ defined by \beqs
M_t & = & \int_0^t \hat Z_s d\hat W_s +\int_0^t \int_\R \hat U_s(e) (N(de,ds)-\hat \nu(de)ds)\;,\quad t\in[0,T]\;,
\enqs
is a local martingale. Using BDG inequality, there exists a constant $C$ such that
\beqs
\E^{\hat\P}\Big[\sup_{t\in[0,T]}|M_t|\Big] & \leq & C\E^{\hat\P}\Big[\sqrt{ \int_0^T |Z_s|^2ds+\int_0^T\int_\R |U_s(e)|^2\hat\nu(de)ds }\Big]
\enqs
From the definition of $\hat \P$, Cauchy-Schwartz inequality and since $k(U^2,U^1)$ is bounded there exists a constant $C'$ such that
\beqs
\E^{\hat\P}\Big[\sup_{t\in[0,T]}|M_t|\Big] & \leq & C'\sqrt{\E\Big[|\Ec_T|^2\Big]}
\sqrt{ \E\Big[\int_0^T |Z_s|^2ds+\int_0^T\int_\R |U_s(e)|^2\nu(de)ds \Big]}~<~+\infty\;.
\enqs
Therefore, $M$ is a martingale under $\hat \P$ as a uniformly integrable local martingale. Taking the conditional expectation given $\Fc_t$ under $\hat \P$ in \reff{ineqdeltaY} give $\hat Y_t  \geq 0$.
 \end{proof}

\bibliographystyle{plain}%apalike}
\bibliography{references}

\begin{thebibliography}{10}

\bibitem{bank2022merton}
Peter Bank and Laura K{\"o}rber.
\newblock Merton's optimal investment problem with jump signals.
\newblock {\em SIAM Journal on Financial Mathematics}, 13(4):1302--1325, 2022.

\bibitem{BS25}
Peter Bank and Gemma Sedrakjan.
\newblock How much should we care about what others know? jump signals in
  optimal investment under relative performance concerns.
\newblock {\em arXiv:2503.16039v2}, 2025.

\bibitem{becherer2006bounded}
Dirk Becherer.
\newblock Bounded solutions to backward sdes with jumps for utility
  optimization and indifference hedging.
\newblock {\em Ann. Appl. Probab.}, 16(4):2027--2054, 2006.

\bibitem{bouchard2008discrete}
Bruno Bouchard and Romuald Elie.
\newblock Discrete-time approximation of decoupled forward--backward sde with
  jumps.
\newblock {\em Stochastic Processes and their Applications}, 118(1):53--75,
  2008.

\bibitem{bouchardmonte}
Bruno Bouchard and Xavier Warin.
\newblock Monte-carlo valorisation of american options: facts and new
  algorithms to improve existing methods.
\newblock {\em Springer Proceedings in Mathematics}, 12.

\bibitem{CK92}
Jaksa Cvitanic and Ioani Karatzas.
\newblock Convex duality in constrained portfolio optimization.
\newblock {\em Ann. Appl. Probab.}, 2:767–818, 1992.

\bibitem{EMN16}
Nicole El~Karoui, Anis Matoussi, and Armand Ngoupeyou.
\newblock {Quadratic Exponential Semimartingales and Application to BSDEs with
  jumps}.
\newblock {\em arXiv:1603.06191v1}, 2016.

\bibitem{Kunita04}
Kunita Hiroshi.
\newblock Representation of martingales with jumps and applications to
  mathematical finance.
\newblock {\em Advanced Studies in Pure Mathematics, Stochastic Analysis and
  Related Topics in Kyoto: In honour of Kiyosi Itô}, 41:209--232, 2004.

\bibitem{hu2005utility}
Ying Hu, Peter Imkeller, and Matthias M{\"u}ller.
\newblock Utility maximization in incomplete markets.
\newblock {\em Ann. Appl. Probab.}, 15(3):1691--1712, 2005.

\bibitem{jacod2013limit}
Jean Jacod and Albert Shiryaev.
\newblock {\em Limit theorems for stochastic processes}, volume 288.
\newblock Springer, A Series of Comprehensive Studies in Mathematics, 2003.

\bibitem{kobylanski}
Magdalena Kobylanski.
\newblock Backward stochastic differential equations and partial differential
  equations with quadratic growth.
\newblock {\em The annals of probability}, 28(2):558--602, 2000.

\bibitem{delong13}
Delong Lukasz.
\newblock {\em {Backward} {Stochastic} {D}ifferenrtial {E}quations with {Jumps}
  and {Their} {Actuarial} and {Financial} {Applications}}.
\newblock Springer-Verlag, 2013.

\bibitem{MS19}
Anis Matoussi and Rym Salhi.
\newblock Exponential quadratic bsdes with infinite activity jumps.
\newblock {\em preprint}, 2020.

\bibitem{Merton69}
Robert~C. Merton.
\newblock {Lifetime Portfolio Selection under Uncertainty: The Continuous-Time
  Case}.
\newblock {\em Journal of Economic Theory}, 3(4):373–413, 1971.

\bibitem{Merton71}
Robert~C. Merton.
\newblock Optimum consumption and portfolio rules in a continuous-time model.
\newblock {\em The Review of Economics and Statistics}, 51(3):373–413,
  247-257.

\bibitem{MAM08}
Marie-Am{\'e}lie Morlais.
\newblock Utility maximization in a jump market model.
\newblock {\em Stochastics and Stochastics Reports}, 80:1--27, 2008.

\bibitem{MAM09}
Marie-Am{\'e}lie Morlais.
\newblock Quadratic bsdes driven by a continuous martingale and application to
  utility maximization problem.
\newblock {\em Finance and Stochastics}, 13:121--150, 2009.

\bibitem{KTPZ15}
Dylan Possamai, Nabil Kazi-Tani, and Chao Zhou.
\newblock {Quadratic BSDEs with jumps: a fixed-point approach}.
\newblock {\em Electron. J. Probab.}, 20:1--28, 2015.

\bibitem{Royer06}
Manuella Royer.
\newblock Backward stochastic differential equations with jumps and related
  non-linear expectations.
\newblock {\em Stochastic Processes and Their Applications},
  116(10):1358--1376, 2006.

\end{thebibliography}

\newpage
%\begin{thebibliography}{9}

%\bibitem{bank2022merton}
%Bank Peter and K{\"o}rber Laura,
%\textit{Merton's Optimal Investment Problem with Jump Signals},
%SIAM Journal on Financial Mathematics, 2022

%\bibitem{jacod2013limit}
%J. Jacod and A. N. Shirayev,
%\textit{Limit theorems for stochastic processes},
%Springer Science \& Business Media, 2013.

%\bibitem{}
%\textit{}

%\end{thebibliography}
\end{document}